%% file: main.tex
\documentclass{amsart}
\usepackage[utf8]{inputenc}
\usepackage[algonl,boxed,norelsize,lined]{algorithm2e}
\usepackage{amscd,amsmath,amsthm,amssymb,amsfonts}
\usepackage{fullpage}
\usepackage{comment}
\usepackage[dvipsnames]{xcolor}
\usepackage[all]{xy}
\usepackage{datetime}
\usepackage{color}
\usepackage[english]{babel}
\usepackage[T1]{fontenc}
\usepackage{hyperref}
\usepackage{enumitem}
\usepackage{theoremref}
\usepackage{enumitem}
\usepackage{tikz-cd}
\usepackage{adjustbox}
\usepackage{tikz-3dplot} 
\usetikzlibrary{arrows}
\usetikzlibrary{3d}
\usepackage{booktabs}

\usepackage[a4paper,top=2.5cm,bottom=2cm,left=2.5cm,right=2.5cm,marginparwidth=1.75cm]{geometry}

\usepackage{amsmath}
\usepackage{amssymb}
\usepackage{amsthm}
\usepackage{verbatim}

\usepackage{tikz}
\usetikzlibrary{decorations.pathreplacing}
\usetikzlibrary{calc}
\usetikzlibrary{patterns}

\newlength{\bibitemsep}
\newlength{\bibparskip}
\let\oldthebibliography\thebibliography
\renewcommand\thebibliography[1]{%
  \oldthebibliography{#1}%
  \setlength{\parskip}{\bibitemsep}%
  \setlength{\itemsep}{\bibparskip}%
}

\newtheorem{Theorem}{Theorem}[section]
\newtheorem*{Theorem*}{Theorem}
\newtheorem*{Corollary*}{Corollary}

\newtheorem{Lemma}[Theorem]{Lemma}
\newtheorem{Corollary}[Theorem]{Corollary}
\newtheorem{Proposition}[Theorem]{Proposition}
\newtheorem{Remark}[Theorem]{Remark}
\newtheorem{Definition}[Theorem]{Definition}
\newtheorem{Example}[Theorem]{Example}

\numberwithin{equation}{section}

\DeclareMathOperator{\link}{link}

\DeclareMathOperator{\conv}{conv}

\DeclareMathOperator{\Star}{Star}

\DeclareMathOperator{\cone}{cone}

\definecolor{caribbeangreen}{rgb}{0.0, 0.8, 0.6}
\definecolor{capri}{rgb}{0.0, 0.75, 1.0}

\title{A combinatorial interpolation between the hypercube and the associahedron}
\author{Lara Bossinger} 
\email{lara@im.unam.mx}
\address{Universidad Nacional Autónoma de México, Instituto de Matemáticas, Unidad Oaxaca, León 2, centro histórico, 68000 Oaxaca de Juárez, México}
\author{Ömer Gürdo\u{g}an}
\email{o.c.gurdogan@soton.ac.uk}
\address{School of Physics and Astronomy, University of Southampton, Highfield, Southampton, SO17 1BJ, United Kingdom.}

\include{pictures}

\begin{document}

\maketitle

\begin{abstract}
We study a family of polytopes interpolating between the hypercube and the associahedron. 
We give an explicit description of their normal fans and describe their combinatorics in terms of chord diagrams associated to their vertices. This yields interesting number sequences generalizing Pell numbers and asymptotic to Catalan numbers.
We describe the transition between the normal fans in terms of star subdivisions and relate them to polytopes associated to Dyck paths as in Calvo Cortes and Frost's work \cite{CalvoCortesFrost2026} and hypergraphic polytopes defined by Benedetti \emph{et al.} in \cite{BenedettiBergeronMachacek2019}. 
Furthermore, all polytopes in our families yield binary geometries.
\end{abstract}


\section{Introduction}

The associahedron is one of the most prominent examples of a polytope whose geometry is governed by a simple combinatorial structure \cite{Stasheff1963I,Stasheff1963II}. 
Its vertices are indexed by triangulations of a polygon, while its faces correspond to collections of pairwise non-crossing diagonals \cite{Lee1989}. 
This description is reflected in the normal fan of the associahedron of dimension $n$, whose rays can be identified with chords of an $(n+3)$-gon. 
The cones of this fan are indexed by collections of non-crossing chords. The associahedron appears in a variety of settings, including cluster algebras, moduli spaces, positive geometry, and the study of scattering amplitudes \cite{FominZelevinsky2003,ArkaniHamedBaiLam2017,ArkaniHamedBaiHeYan2018}.

In this paper we study a family of polytopes that interpolates between the hypercube and the associahedron. 
The starting point for our construction is the pellytope introduced by the first author, Telek and Tillmann-Morris in \cite{Pellytopes}. 
Its normal fan is a coarsening of the normal fan of the associahedron. 
This raises a natural question: can one construct a family of polytopes interpolating between these two types of objects?
And is there a combinatorial description analogous to the triangulations associated with vertices of the associahedron?

We answer these questions by introducing a family of polytopes $\mathcal{P}_{d,n}$. We refer to them as \emph{Catalan pellytopes} as they generalise pellytopes all the way to associahedra following Catalan combinatorics.
Let $n\geq 1$ and consider the polynomials
\begin{eqnarray}\label{eq:qkj}
    q_{k,j}=1+y_j+y_jy_{j+1}+\cdots+y_jy_{j+1}\cdots y_{j+k-1},
\end{eqnarray}
where $1\leq k\leq n$ and $1\leq j\leq n-k+1$. Let $Q_{k,j}:=\operatorname{Newt}(q_{k,j})\subseteq \mathbb{R}^n$ be their Newton polytopes. 
For $1\leq d\leq n$, we define the {\bf Catalan pellytope}
\begin{eqnarray}\label{eq:def Catalan pellytope}
\mathcal P_{d,n}
:= \operatorname{Newt}
\left( \prod_{k=1}^{d} \prod_{j=1}^{n-k+1} q_{k,j} \right) = \sum_{k=1}^{d} \sum_{j=1}^{n-k+1} Q_{k,j}.    
\end{eqnarray}
where the right hand side is the Minkowski sum of the corresponding Newton polytopes.
The polytopes we study are examples of polytopes associated to Dyck paths \cite{CalvoCortesFrost2026}, which themselves are examples of  hypergraphic polytopes--Minkowski sums of faces of a simplex \cite{BenedettiBergeronMachacek2019}. 
This places them in the class of generalized permutohedra, which also contains the classical associahedron \cite{ArdilaBenedettiDoker2010,Postnikov_associahedra}.

\begin{figure}
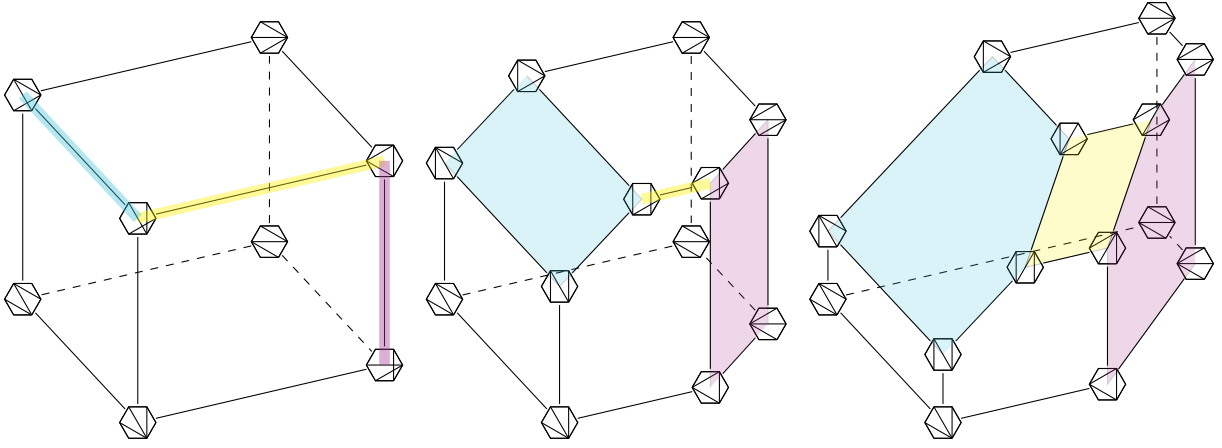

  \centering
  \picpone\hfill
  \picptwo\hfill
  \picpthree
    \caption{The Catalan pellytopes $\mathcal P_{1,3},\mathcal P_{2,3}$ and $\mathcal P_{3,3}$ with vertices labeled by chord diagrams by Theorem~\ref{thm:chord fan}. 
    The bottom front chord diagram is the triangulation with all chords incident to vertex $1$ of the hexagon.} 
    \label{fig:placeholder}
\end{figure}

Special cases of this construction recover several familiar polytopes. For $d=1$, each $Q_{1,j}$ is a line segment and $\mathcal P_{1,n}$ is the $n$-dimensional hypercube. 
For $d=2$, $\mathcal P_{2,n}$ is the pellytope introduced in \cite{Pellytopes}. 
At the other extreme, $\mathcal P_{n,n}$ is unimodularly equivalent to the Loday realization of the associahedron $\operatorname{Ass}_n$, see Corollary~\ref{cor:ass}. 

The parameter $d$ has a direct interpretation in the normal fan. Denote the inner normal fan of $\mathcal P_{d,n}$ by $\Sigma_{d,n}$.
We show in Proposition~\ref{cor:rays} that the rays of $\Sigma_{d,n}$ are generated by
\[
\pm e_i
\qquad\text{and}\qquad
e_i-e_{i+k} \text{ for }  1\leq k\leq d-1.
\]
These vectors can be naturally identified with a collection of chords in an $(n+3)$-gon: 
\begin{eqnarray}\label{eq:chords and rays}
    -e_a \leftrightarrow \overline{1(a+2)}, \quad e_b \leftrightarrow \overline{(b+1)(n+3)}, \quad e_i-e_j \leftrightarrow \overline{(i+1)(j+2)}
\end{eqnarray}
for $1\le a,b\le n$, $1\le i\le n-1$ and $i+1\le j\le n$.
This identification is induced by interpreting the ray generators as $g$-vectors of cluster variables associated to a linear orientation of a type $\mathtt A$ Dynkin diagram \cite{FominZelevinsky-I,FominZelevinsky-IV}, see \eqref{eq:chords and type A quiver}.
We denote the collection of chords corresponding to rays of $\Sigma_{d,n}$ by $\mathcal{C}h_d$ and refer to its elements as \emph{$d$-chords}. 
Two $d$-chords are {\it $d$-compatible}, if either they do not cross, or they are of form $\overline{1j},\overline{i(n+3)}$ and $j-i\ge d+1$.
A maximal collection of $d$-compatible $d$-chords is called a \emph{$d$-admissible chord diagram}, see Definition~\ref{eq:def d-chords} and Figure~\ref{fig:exp chord diag}.
In particular, for $d=n$, $d$-chords are the usual diagonals of the $(n+3)$-gon and $d$-admissible chord diagrams are precisely triangulations.

Our main result identifies chord diagrams with the maximal cones of the normal fan of the Catalan pellytope.

\begin{Theorem*}[Theorem~\ref{thm:chord fan} and Proposition~\ref{prop:stars}]
Let $1\leq d\leq n$. The inner normal fan $\Sigma_{d,n}$ of the Catalan pellytope $\mathcal P_{d,n}$ is the \emph{chord diagram fan} $\mathcal{C}_{d,n}$ whose maximal cones are indexed by $d$-admissible chord diagrams. 
More precisely:
\begin{enumerate}
    \item A collection of rays of $\Sigma_{d,n}$ spans a cone if and only if the corresponding $d$-chords are pairwise $d$-compatible. In particular, the maximal cones of $\Sigma_{d,n}$ are in bijection with $d$-admissible chord diagrams.
    \item The fan $\Sigma_{d,n}$ is smooth, in particular simplicial, and flag.
\end{enumerate}
Moreover, $\Sigma_{d,n}$ is obtained from $\Sigma_{d-1,n}$ by a sequence of star subdivisions along two-dimensional cones corresponding to crossing chords 
\[
\overline{1j}\text{ and }\overline{i(n+3)}, \text{ with } j-i=d.
\]
Each star subdivision introduces a new ray corresponding to the $d$-chord $\overline{ij}$.
\end{Theorem*}

Figure~\ref{fig:exp lift} displays an example. 

The local structure of these fans is also inherited from the associahedron. A chord decomposes the $(n+3)$-gon into two smaller polygons, and the link of the corresponding ray decomposes as a product of two smaller chord diagram fans. In particular, the recursive decomposition of the associahedron along a diagonal persists throughout the Catalan pellytope family, although the parameter $d$ is redistributed between the two resulting subpolygons as follows.

\begin{Theorem*}[Proposition~\ref{prop:link}]
    For every $\overline{ij}\in\mathcal{C}h_d$, there exist $d_1,d_2,n_1,n_2$ such that the link at the ray of $\overline{ij}$ is the product of two smaller fans $\Sigma_{d_1,n_1}\times\Sigma_{d_2,n_2}$,
    where $n_1,n_2\ge 0$ such that $n_1+n_2=n-1$ and for each $i\in\{1,2\}$ either $(d_i,n_i)=(0,0)$ or $d_i\in[n_i]$ with $1\le d_1+d_2\le n-1$.
\end{Theorem*}

\begin{table}
\[
\begin{array}{c|rrrrrrrrrr}
n\backslash d & 1&2&3&4&5&6&7&8&9&10\\
\hline
1  & 2&&&&&&&&&\\
2  & 4&5&&&&&&&&\\
3  & 8&12&14&&&&&&&\\
4  & 16&29&37&42&&&&&&\\
5  & 32&70&98&118&132&&&&&\\
6  & 64&169&261&331&387&429&&&&\\
7  & 128&408&694&934&1130&1298&1430&&&\\
8  & 256&985&1845&2645&3317&3905&4433&4862&&\\
9  & 512&2378&4906&7476&9786&11802&13650&15366&16796&\\
10 & 1024&5741&13045&21120&28932&35862&42198&48204&53924&58786
\end{array}
\]
\caption{The values of $S_n^{(d)}$ for $n\le 10$ and $d\le n$. Notice that the Catalan numbers $C_k=S_{k-1}^{(k-1)}$ appear along the diagonal.}
\label{tab:S_n^d}
\end{table}

The recursive nature of the chord diagram description has an immediate enumerative consequence. Let $S_n^{(d)}$ denote the number of vertices of $\mathcal P_{d,n}$.
The first few numbers are shown in Table~\ref{tab:S_n^d}.
Since the vertices are indexed by $d$-admissible chord diagrams, these numbers are obtained by counting such diagrams. 
For $d=2$, this gives the Pell numbers, while for $d=n$ it gives the Catalan numbers. 
For intermediate values of $d$, one obtains a family of sequences generalizing Pell numbers and approaching Catalan numbers as $d\to\infty$. 
The corresponding generating functions and recurrence relations are studied in \S\ref{sec:numbers}. 
We show in Proposition~\ref{prop:recursion} that for fixed $d$ the numbers satisfy a $d$-term recursion.
This leads us to the definition of {\bf Catalan Pell numbers} given by the recursion
\[\textstyle
\mathcal S_n^{(d)}:= 2\mathcal S_{n-1}^{(d)} + \sum_{k=1}^{d-1} C_k \mathcal S_{n-k-1}^{(d)},
\]
for $d,n\ge 0$ subject to the initial condition $\mathcal S_n^{(d)}=C_{n+1}$ for $n\le d$ where $C_{n}$ is the $n^\text{th}$ {Catalan number}.

\begin{Theorem*}[Corollary~\ref{cor:generating function} and Equation~\ref{eq:summary numbers}]
Catalan Pell numbers are given by the generating function 
\begin{equation*}\label{eq:genfunc}\textstyle
\mathcal S^{(d)}(x) = \left(1-2x- \sum_{k=1}^{d-1}C_{k} x^{k+1}\right)^{-1}.
\end{equation*}
Moreover, $\mathcal S_n^{(1)}=2^n$, $\left(\mathcal S_n^{(2)}\right)_n=\left(P_{n+1}\right)_n$ is the sequence of  \emph{Pell numbers}, and $\lim_{d\to \infty} \left(\mathcal S_n^{(d)}\right)_n=(C_{n+1})_n$.
\end{Theorem*}

\bigskip
\noindent
{\bf Application.}
Binary geometries are affine varieties with stratifications determined by certain simplicial complexes, see Definition~\ref{def:BinaryGeom}.
These geometric objects, like positive geometries, first arose from the study of canonical forms on polytopes and amplituhedra, as a novel method to compute scattering amplitudes \cite{Arkani-Hamed:2013jha,ArkaniHamedBaiLam2017}.
The normal fans of pellytopes and associahedra are examples of simplicial complexes that yield binary geometries \cite{Pellytopes,AHL::Stringy}.
It is therefore natural to ask whether the same is true for general $\Sigma_{d,n}$.

Binary geometries are cut out by polynomial equations called $u$-equations, one for each ray in the fan.
For a $d$-chord $\overline{ij}$ the corresponding $u$-equation takes the form
\[
u_{ij} + \prod_{\overline{pq}\text{ not $d$-compatible to } \overline{ij}}u_{pq}  = 1.
\]
A positive parametrization of this binary geometry is expressed in terms of the polynomials $y_i$ and $q_{k,j}$ that define the Catalan pellytope. 
The results of this application rely on the fact that our polytopes coincide with polytopes associated to Dyck paths \cite{CalvoCortesFrost2026}.
In \emph{loc.cit.} they provide the relevant positive parametrization, while this paper gives an explicit inverse to it in the form of a chamber Ansatz following Berenstein--Fomin--Zelevinsky \cite{parametrizations}.
This enables us to verify that indeed each $\Sigma_{d,n}$ determines a binary geometry, see Proposition~\ref{prop:binary}. 

Several maps related to the chamber Ansatz, such as the dual map and the inverse map have been studied extensively, for example in \cite{Lara_Ghislain,Genz_typeA}.
The explicit diagram we use for the chamber Ansatz is depicted in Figure~\ref{fig:dyck diagram}. 
It resembles the graph of a Gelfand--Tsetlin pattern.

\medskip
\noindent
{\bf Structure.}
The paper is organized as follows. In \S\ref{sec:pre} we recall the necessary background on polytopes and simplicial complexes. 
In \S\ref{sec:polytopes} we introduce the Catalan pellytopes, connect to hypergraphic polytopes, and reveal the connection to the combinatorics of chord diagrams. We also give the connection to Dyck paths and the chamber Ansatz.
In \S\ref{sec:binary} we study the application to binary geometries.

\medskip
\noindent
{\bf Acknowledgements.} We would like to thank Veronica Calvo Cortes and Hadleigh Frost for explaining their work to us and for stimulating discussions.
We are grateful to the Erwin Schrödinger International Institute for Mathematics and Physics (ESI) of the University of Vienna for hosting the programme "Algebra and Geometry of Amplitudes" organized by Daniele Agostini, the first author, Ruth Britto, Johannes Henn, Jianrong Li, Anna-Laura Sattelberger and Oliver Schlotterer. This project started during the programme.
LB is grateful for the support of UNAM DGAPA PAPIIT IN106126 and SECIHTI CF-2023-G-106.
\"OCG acknowledges the support of the Royal Society University Research Fellowship URF\textbackslash R1\textbackslash221236 as well as the STFC consolidated grant ST/X000583/1.

\medskip
\noindent
{\bf Declaration.} In the preparation of this manuscript, the authors used LLMs for language editing and assistance with exposition. 
All mathematical arguments, results, and references were independently verified by the authors, who take full responsibility for the content of the manuscript.

\section{Preliminaries}\label{sec:pre}
Let $[n] = \{1,\dots,n\}$. A {\bf simplicial complex} $\Delta$ on $[n]$ is a non-empty collection of subsets of $[n]$ satisfying the following property:
\begin{itemize}
    \item[(i)] If $S \in \Delta$ and $S' \subset S$ then $S' \in \Delta$.
\end{itemize}
If, additionally, a simplicial complex $\Delta$ satisfies:
\begin{itemize}
    \item[(ii)] $\{k\} \in \Delta$ for each $k \in [n]$ and 
    { \item[(iii)] if $\{k_1,\dots,k_r\} \subset  [n]$ with $\{k_i,k_j\} \in \Delta$ for all $1\leq i < j \leq r$, then $\{k_1,\dots,k_r\} \in \Delta$}
\end{itemize}
we call $\Delta$ a {\bf flag complex}. A simplicial complex $\Delta$ is {\bf pure} if all the maximal sets in $\Delta$ (with respect to inclusion) have the same cardinality. We say that $i, j \in [n]$ are {\bf incompatible} if $\{i, j\} \notin \Delta$, and write $i \nsim j$.
Denote by $|S|$ the cardinality of the set $S$.
Given a simplicial complex $\Delta$ define the {\bf link} $\link_\Delta S$ of $S \in \Delta$ as the simplicial complex 
\begin{eqnarray}\label{eq:link}
    \link_\Delta S := \big\{ T \in \Delta \mid S \cap T = \emptyset \text{ and } S \cup T \in \Delta \big\} 
\end{eqnarray} 
on the set $W_S = \{j \in [n] \setminus S \mid S \cup \{ j \} \in \Delta\}$. In other words, the set $W_S$ contains all $j$ that are compatible with $S$ but do not lie in $S$.

Let $p_1,\dots,p_r\in \mathbb Z^n$. Then the {\bf convex hull} of $p_1,\dots,p_r$ is
\[
\conv\left(p_1,\dots,p_r\right):=\left\{\sum_{i=1}^r \lambda_ip_i:\lambda_i\in \mathbb R_{\ge 0}, \sum_{i=1}r\lambda_i=1\right\}.
\]
The convex hull of a finite number of lattice points is a {\bf (lattice) polytope}. Two lattice polytopes are {\bf unimodularly equivalent} if there exists a lattice isomorphism (given by a unimodular matrix) that sends one polytope to the other.
Unimodular equivalence preserves the combinatorial structure of the polytopes as well as the number of lattice points and the volume (relative to the lattice).
Recall the definition of {\bf Minkowski sum} of two polytopes $P,Q\subset \mathbb R^n$
\[
P+Q:=\{p+q:p\in P,q\in Q\}.
\]
The polytopes in this paper arise as Newton polytopes of Laurent polynomials:
Let $k$ be a field and let $S^{\pm}=k[x_1^{\pm 1},\dots,x_n^{\pm 1}]$. For $a=(a_1,\dots,a_n)\in\mathbb Z^n$ denote by $x^a$ the monomial $x_1^{a_1}\cdots x_n^{a_n}$. 
Given a polynomial $f=c_1x^{a^{(1)}}+c_2x^{a^{(2)}}+\dots +c_rx^{a^{(r)}}$ with $a^{(i)}\in\mathbb Z^n$ and $c_i\in k\setminus \{0\}$ for $i\in[r]$ define its {\bf Newton polytope} as 
\[
\operatorname{Newt}(f):=\conv\left(a^{(1)},\dots,a^{(r)}\right).
\]
Recall that for $f,g\in S$ the Newton polytope of their product $fg$ can be decomposed as Minkowski sum
\[
\operatorname{Newt}(fg)=\operatorname{Newt}(f)+\operatorname{Newt}(g).
\]
In particular, if $f=x^a$ is a monomial, the Newton polytope of the product $x^a g$ is the translate by $a$:
\begin{eqnarray}\label{eq:mink sum monomial}
\operatorname{Newt}(x^a g)=\operatorname{Newt}(g) + a.
\end{eqnarray}
The {\bf conic hull} of $p_1,\dots,p_r$, also called the {\bf cone generated by} $p_1,\dots,p_r$ is
\[
\cone\left(p_1,\dots,p_r\right):=\left\{\sum_{i=1}^r \lambda_ip_i:\lambda_i\in \mathbb R_{\ge 0}\right\}.
\]
For a polytope $P\subset \mathbb R^n$ denote by $\Sigma(P)$ its inner normal fan and by $V(P)$ its set of vertices.
Given two fans $\Sigma,\Sigma'\subset \mathbb R^n$ recall their {\bf common refinement}
\[
\Sigma\wedge\Sigma'=\{C\cap C':C\in \Sigma, C'\in \Sigma'\}.
\]
Moreover, 
\begin{equation}\label{eq:refinement}
    \Sigma(P+Q)=\Sigma(P)\wedge \Sigma(Q).
\end{equation}
A fan $\Sigma$ is called {\bf strictly convex} if every cone is strictly convex, that is it does not contain a linear subspace. 
If $C\subset \mathbb R^n$ is a cone containing a linear subspace $\mathcal L$, its quotient $C/\mathcal L\subset \mathbb R^n/\mathcal L$ is strictly convex and hence has unique rays $\bar \rho_1,\dots,\bar \rho_s\subset \mathbb R^n/\mathcal L$.
Let $\rho_1,\dots,\rho_s\subset \mathbb R^n$ be representatives of $\bar \rho_1,\dots,\bar \rho_s$. 
By a little abuse of notation we will refer to $\rho_1,\dots,\rho_s$ as the rays of $C$, keeping in mind that they are not unique and may be replaced by any other collection $\{\rho_j+\ell_j:j\in[s],\ell_j\in \mathcal L\}$.

Let $(\mathbb R^n)^*$ denote the vector space dual to $\mathbb R^n$ with pairing
\[
\langle \cdot,\cdot \rangle :\mathbb R^{n}\times (\mathbb R^n)^* \to \mathbb R
\]
Given $p\in (\mathbb R^n)^*, a\in \mathbb R$ define the following subsets of $\mathbb R^n$:
\begin{eqnarray*}
H_{p,a}&:=&\{q\in\mathbb R^n:\langle p,q\rangle =a \},\\
H_p^+&:=&\{q\in\mathbb R^n:\langle p,q\rangle \ge0 \},\\
H_p^-&:=&\{q\in\mathbb R^n:\langle p,q\rangle \le0 \}.
\end{eqnarray*}
Set $H_{p}:=H_{p,0}$. Fix a basis $\{e_1,\dots,e_n\}$ for $(\mathbb R^n)^*$ and dual basis $\{f_1,\dots,f_n\}$ for $\mathbb R^n$.

All cones in this paper are assumed to be {\bf rational} cones, that is they are spanned by elements in $\mathbb Z^n$.
For a cone $\sigma\subset \mathbb R^n$ we denote by $\sigma(1)$ its set of rays.
For $\rho\in \sigma(1)$ denote by $u_\rho\in \mathbb Z^n$ its {\bf primitive ray generator}, that is lattice point in $\rho$ closest to the origin.
For a fan $\Sigma\subset \mathbb R^n$ denote by $\Sigma(d)$ the set of its $d$-dimensional cones.

Let $\sigma\subset \mathbb R^n$ be an $n$-dimensional simplicial cone with primitive ray generators $v_1,\dots,v_n$. 
Then $\sigma$ is {\bf smooth} if the determinant of the matrix with columns $v_1,\dots,v_n$ is $\pm1$. 
Said differently, $v_1,\dots,v_n$ form a lattice basis.

Let $\Sigma\subset \mathbb R^n$ be an $n$-dimensional fan and $\sigma\in \Sigma$ a smooth cone with primitive ray generators $v_1,\dots,v_n$ and let $v_0=v_1+\dots+v_n$.
Define $\Sigma'(\sigma)$ as the fan consisting of all cones generated by subsets of $\{v_0,\dots,v_n\}$ not containing $\{v_1,\dots,v_n\}$. Then the {\bf star subdivision of $\Sigma$ relative to $\sigma$} is defined as
\[
\operatorname{Star}_\sigma(\Sigma)= (\Sigma\setminus \sigma)\cup \Sigma'(\sigma).
\]

Let $\tau\in\Sigma$ be a cone with the property that all maximal cones containing $\tau$ are smooth. In particular, $\tau$ is simplicial.
Define $u_\tau=\sum_{\rho\in \tau(1)}u_\rho$.
For each cone $\sigma\in \Sigma$ with $\tau\subseteq \sigma$ define
\[
\Sigma_\sigma^*(\tau) =\{\cone(A):A\subset \sigma(1)\cup\{u_\tau\}, \tau(1)\not\subseteq A\}.
\]
Then the {\bf star subdivision of $\Sigma$ relative to $\tau$} is
\begin{equation}\label{eq:starsub}
\operatorname{Star}_\tau(\Sigma) = \{\sigma'\in \Sigma: \tau\not\subseteq \sigma'\}\cup \bigcup_{\sigma \supseteq \tau} \Sigma_\sigma^*(\tau).
\end{equation}

\begin{Lemma}
    Let $\Sigma\subset \mathbb R^n$ be an $n$-dimensional fan with $r$ rays and $s$ maximal cones and let $\tau\in\Sigma$ be a $k$-dimensional cone with the property that all maximal cones containing $\tau$ are smooth. Assume there are $t$ such maximal cones.
    Then $\Star_\tau(\Sigma)$ is a refinement of $\Sigma$ with $r+1$ rays and $s+t(k-1)$  maximal cones.
\end{Lemma}

\begin{proof}
    Each of the $t$ cones $\sigma\in \Sigma(n)$ with $\tau\subset\sigma$ in $\Sigma$ is replaced by $\Sigma_\sigma^*(\tau)$, which is the $n$-dimensional fan obtained from $\sigma$ by replacing the face $\tau$ by the $k$-dimensional fan $\Sigma'(\tau)$ which has $k-1$ maximal cones.
\end{proof}

We will need the following criterion regarding commutativity of star subdivisions

\begin{Lemma}\label{lem:commuting stars}
Let $\Sigma$ be a simplicial fan and let $\tau,\tau'\in\Sigma$
be smooth cones with disjoint rays.
Then the star subdivisions at $\tau$ and $\tau'$ commute:
\[
\Star_{\tau'}\left(\Star_{\tau}(\Sigma)\right)=\Star_{\tau}\left(\Star_{\tau'}(\Sigma)\right).
\]
\end{Lemma}

\begin{proof}
It suffices to consider a maximal cone $\sigma\in\Sigma$ containing
both $\tau$ and $\tau'$, since otherwise the two subdivisions act
independently on $\sigma$. Write
\[
\tau(1)=\{\rho_1,\ldots,\rho_k\}, \qquad \tau(1)=\{\rho'_1,\ldots,\rho'_{k'}\}.
\]
Since $\Sigma$ is simplicial and $\tau(1)\cap\tau'(1)=\varnothing$,
the part of either iterated subdivision supported on $\sigma$
consists of the cones
\[
\cone\left(
u_\tau,u_{\tau'},
\sigma(1)\setminus\{\rho_i,\rho'_j\}
\right),
\qquad
1\le i\le k,\quad 1\le j\le k',
\]
together with their faces. This description is symmetric in
$\tau$ and $\tau'$, so the resulting fan is independent of the
order of subdivision.
\end{proof}

\section{Catalan pellytopes and chord diagrams}\label{sec:polytopes}
Fix $n\in \mathbb Z_{\ge 0}$.
Denote by $S_y$ the polynomial ring $k[y_1,\dots,y_n]$ over a field $k$.  
We interpret $y_1,\dots,y_n$ as monomials with exponent vector $f_1,\dots,f_n\in \mathbb R^n$.
Recall, from the introduction, the polynomials $q_{k,j}$ defined in \eqref{eq:qkj}, their Newton polytopes $Q_{k,j}$ and the Catalan pellytopes $\mathcal P_{d,n}$ from \eqref{eq:def Catalan pellytope}.
Namely, for a fixed $n$ we have that $\mathcal P_{d,n}$ is the Minkowski sum of all $Q_{k,j}$ up to a fixed degree $k\le d$.
The family of Catalan pellytopes is a subclass of hypergraphic polytopes \cite{BenedettiBergeronMachacek2019}, as we show in Proposition~\ref{prop:hypergraph poly isom}. As such, it contains three distinguished polytopes depending on the parameter $d$.

\begin{Example}
    When $d=1$ the Catalan pellytope $\mathcal P_{1,n}$ is the Minkowski sum of the line segments $\operatorname{Newt}(1+y_i)$ for all $i\in[n]$. Hence, $\mathcal P_{1,n}$ is the hypercube $\square_n$.
\end{Example}

\begin{Example}
    For $d=2$ the Catalan pellytope $\mathcal P_{2,n}$ is, by definition, the {\bf pellytope} studied in \cite{Pellytopes}.
\end{Example}

Moreover, the associahedron is contained in this family, see Corollary~\ref{cor:ass} below. 
To make this precise we realize Catalan pellytopes as hypergraphic polytopes.
Consider the map between the Laurent polynomial ring $S_y^{\pm}=k[y_1^{\pm 1},\dots, y_n^{\pm 1}]$ and degree zero elements in $S_t^\pm=k[t_0^{\pm 1},\dots,t_n^{\pm 1}]$ given by
\begin{eqnarray}\label{eq:laurent isom}
    y_i\mapsto \frac{t_i}{t_{i-1}}.
\end{eqnarray}
Then $q_{k,i}\mapsto \frac{t_{j-1}+t_j+\dots t_{j+k-1}}{t_{j-1}}$.
Let $A_n$ be the lattice $\{(a_0,\dots,a_n)\in \mathbb Z^{n+1}:\sum_{i=0}^n a_i=0\}$. 
A basis for $A_n$ is given by $\epsilon_i-\epsilon_{i-1}$ for $i=1,\dots,n$, where $\epsilon_0,\dots,\epsilon_n$ is the standard basis of $\mathbb Z^{n+1}$.
Furthermore, consider $\mathbb Z^n$ with standard basis $e_i$ for $i=1,\dots,n$.
The map between the Laurent polynomial rings induces a lattice map $e_i\mapsto \epsilon_i-\epsilon_{i-1}$ for $i=1,\dots,n$ and defines a lattice isomorphism $\varphi:\mathbb Z^n \to A_n$. 
It is given by an upper triangular matrix with $1$s on the diagonal, hence it is a unimodular lattice transformation.
For the simplices $Q_{k,j}$ we have
\[
\varphi(Q_{k,j})= \conv(\epsilon_{j-1},\epsilon_j,\dots,\epsilon_{j+k-1}) - \epsilon_{j-1}.
\]
So, up to a translation, we have a lattice isomorphism between $Q_{k,j}$ and the simplex $\Delta_{[j-1,j+k-1]}:=\conv(\epsilon_{j-1},\epsilon_j,\dots,\epsilon_{j+k-1})$.
More generally, for a subset $I\subseteq [0,n]$ define the simplex $\Delta_I$ as the convex hull of $\{\epsilon_i:i\in I\}$. 
The following definition is due to Benedetti, Bergeron and Machacek \cite{BenedettiBergeronMachacek2019} building on work of Postnikov \cite{Postnikov_associahedra}, and Ardila, Benedetti and Doker \cite{ArdilaBenedettiDoker2010}.

\begin{Definition}
Given a hypergraph $\mathcal H$ with vertex set $[0,n]$ and hyperedges $\mathcal E$ the associated {\bf hypergraphic polytope} $\mathcal P_{\mathcal H}$ is defined as the Minkowski sum
\[
\mathcal P_{\mathcal H}= \sum_{I\in\mathcal E} \Delta_I.
\]
\end{Definition}

Let $\mathcal H_{d,n}$ be the hypergraph with vertex set $[0,n]$ and hyperedges are all intervals $I$ with $2\le |I|\le d+1$. Denote the set of hyperedges by $\mathcal E_{d,n}$.

\begin{Example}
For $n=3$ we have the three hypergraphs $\mathcal H_{1,3},\mathcal H_{2,3}$ and $\mathcal H_{3,3}$. They all have vertex set $\{0,1,2,3\}$ and their hyperedges $\mathcal E_{d,3}$ are:
\[
\mathcal E_{1,3}=\{[0,1],[1,2],[2,3]\}, \quad  \mathcal E_{2,3}=\mathcal E_{1,3}\cup\{[0,2],[1,3]\},\quad \mathcal E_{3,3}=\mathcal E_{2,3}\cup\{[0,3]\}. 
\]
\end{Example}

\begin{Proposition}\label{prop:hypergraph poly isom}
    The Catalan pellytope $\mathcal P_{d,n}$ is unimodularly equivalent to the hypergraphic polytope associated with the hypergraph $\mathcal H_{d,n}$.
\end{Proposition}
\begin{proof}
We show that $\mathcal P_{d,n}$ and $\mathcal P_{\mathcal H_{d,n}}$ are related by the unimodular lattice transformation $\varphi:\mathbb Z^n\to A_n$ induced by \eqref{eq:laurent isom}.
We have
\[
\mathcal P_{d,n}=\operatorname{Newt}\left(\prod_{k=1}^d\prod_{j=1}^{n-k+1}  q_{k,j}\right) \mapsto \operatorname{Newt}\left( \prod_{k=1}^d\prod_{j=1}^{n-k+1} \frac{t_{j-1}+t_j+\dots t_{j+k-1}}{t_{j-1}}\right)
\]
Notice that we can modify the indexing to see that 
\[
 \operatorname{Newt}\left( \prod_{k=1}^d\prod_{j=1}^{n-k+1} \frac{t_{j-1}+t_j+\dots t_{j+k-1}}{t_{j-1}}\right)
= \mathcal P_{\mathcal H_{d,n}}+m
\]
for $m\in\mathbb Z^{n+1}$ the exponent of the monomial in the denominator.
\end{proof}

For $d=n$, this recovers the Loday realization of the associahedron \cite{Postnikov_associahedra}.
In Section 8, immediately after Corollary 8.2, on p. 1052 in {\it loc.cit.} the associahedron is given as
\begin{eqnarray}\label{eq:associahedron}
\operatorname{Ass}_n
:=
\operatorname{Newt}
\left(
\prod_{0\leq i\leq j\leq n}
(t_i+t_{i+1}+\cdots+t_j)
\right)
\end{eqnarray}
where $t_i+t_{i+1}+\cdots+t_j$ is a polynomial in $S_t=k[t_0^{\pm 1},\dots,t_n^{\pm 1}]$. 
So, $\operatorname{Ass}_n$ coincides with $\mathcal P_{\mathcal H_{n,n}}$ up to translation.

\begin{Corollary}\label{cor:ass}
    The Catalan pellytope $\mathcal P_{n,n}$  is unimodularly equivalent to the associahedron $\operatorname{Ass}_n$.
\end{Corollary}

We need the normal fan of each simplex before we proceed.

\begin{Lemma}\label{lem:rays Q}
The inner normal fan $\Sigma(Q_{d,j})$ has rays generated by
\begin{eqnarray}\label{eq:normal vectors simplex Qdi}
-e_j, e_j-e_{j+1}, e_{j+1}-e_{j+2},\dots, e_{j+d-2}-e_{j+d-1},e_{j+d-1}
\end{eqnarray}
and lineality space $\mathcal L_{d,j} = \operatorname{span}_{\mathbb R}\{e_p:p\notin [j,j+d-1]\}$. 
\end{Lemma}

\begin{proof} The polytope $Q_{d,j}$ is contained in the subspace 
\[ 
V_{d,j} = \operatorname{span}_{\mathbb R} \{f_j,f_{j+1},\ldots,f_{j+d-1}\} 
\] 
and has vertices $ 0, f_j, f_j+f_{j+1}, \ldots, f_j+\cdots+f_{j+d-1}$. 
Thus $Q_{d,j}$ is a $d$-dimensional simplex in $V_{d,j}$. 
Writing $ x=\sum_{r=j}^{j+d-1}x_rf_r, $ the simplex is described by 
\[ 1\geq x_j\geq x_{j+1}\geq\cdots \geq x_{j+d-1}\geq 0. 
\] 
Consequently, the facets of $Q_{d,j}$ have the claimed primitive inward normal vectors.
Since $Q_{d,j}$ is contained in $V_{d,j}$, every vector $e_p$ with $p\notin[j,j+d-1]$ annihilates $Q_{d,j}$.
Hence, the lineality space of its normal fan is as claimed.
\end{proof}

\begin{figure}
    \centering
\begin{tikzpicture}[scale=.75]
 \coordinate (2) at (1.848, 0.765);   
    \coordinate (1) at (0.765, 1.848);   
    \coordinate (8) at (-0.765, 1.848);
    \coordinate (7) at (-1.848, 0.765);
    \coordinate (6) at (-1.848, -0.765);
    \coordinate (5) at (-0.765, -1.848);
    \coordinate (4) at (0.765, -1.848);
    \coordinate (3) at (1.848, -0.765);

    \draw[thick] (1) -- (2) -- (3) -- (4) -- (5) -- (6) -- (7) -- (8) -- cycle;
\draw[red,thick] (6) -- (8) -- (5);
\draw[red,thick] (8) -- (2) -- (5);
\draw[red, thick] (4) -- (2);

    \node[right] at (2) {2};
    \node[above right] at (1) {1};
    \node[above] at (8) {8};
    \node[above left] at (7) {7};
    \node[left] at (6) {6};
    \node[below left] at (5) {5};
    \node[below] at (4) {4};
    \node[below right] at (3) {3};

    \begin{scope}[xshift=7cm]
 \coordinate (2) at (1.848, 0.765);   
    \coordinate (1) at (0.765, 1.848);   
    \coordinate (8) at (-0.765, 1.848);
    \coordinate (7) at (-1.848, 0.765);
    \coordinate (6) at (-1.848, -0.765);
    \coordinate (5) at (-0.765, -1.848);
    \coordinate (4) at (0.765, -1.848);
    \coordinate (3) at (1.848, -0.765);

    \draw[thick] (1) -- (2) -- (3) -- (4) -- (5) -- (6) -- (7) -- (8) -- cycle;
\draw[red,thick] (7) -- (1) -- (6);
\draw[red,thick] (8) -- (2) -- (5);
\draw[red,thick] (2) -- (4);

    \node[right] at (2) {2};
    \node[above right] at (1) {1};
    \node[above] at (8) {8};
    \node[above left] at (7) {7};
    \node[left] at (6) {6};
    \node[below left] at (5) {5};
    \node[below] at (4) {4};
    \node[below right] at (3) {3};
    \end{scope}
\end{tikzpicture}
    \caption{Two chord diagrams for $n=5$. The left diagram is $d$-admissible for $d\ge 3$, the right chord diagram is $3$-admissible, but not $4$-admissible due to the crossing chords $\overline{28}$ and $\overline{16}$ that are not $4$-compatible.}
    \label{fig:exp chord diag}
\end{figure}
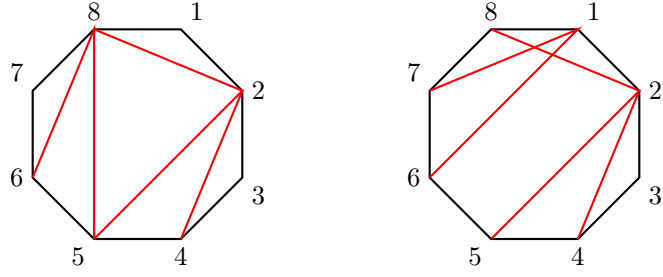

\subsection{Chord diagrams}\label{sec:chords}
Let $Q$ be the equioriented quiver of type $\mathtt A_n$
\[
v_1\to v_2\to \dots\to v_{n-1}\to v_n
\]
and consider the corresponding cluster algebra $A$ with principal coefficients with respect to the seed associated with $Q$ \cite{FominZelevinsky-I,FominZelevinsky-IV}.
We identify $Q$ with the triangulation of an $(n+3)$-gon where all chords start at the vertex $1$. 
So we have a bijection between chords of the initial triangulation and vertices of $Q$:
\begin{eqnarray}\label{eq:chords and type A quiver}
v_1 \leftrightarrow \overline{13}, \quad v_2 \leftrightarrow \overline{14},\quad\dots, \quad v_{n-1} \leftrightarrow \overline{1,n+1}, \quad v_n\leftrightarrow \overline{1,n+2}.
\end{eqnarray}
We identify the chords with their $g$-vectors and will see subsequently that these are the ray generators of $\Sigma_{n,n}$, see Corollary~\ref{cor:rays}.
\begin{eqnarray}\label{eq:chords and rays}
    -e_a \leftrightarrow \overline{1(a+2)}, \quad e_b \leftrightarrow \overline{(b+1)(n+3)}, \quad e_i-e_j \leftrightarrow \overline{(i+1)(j+2)}
\end{eqnarray}
for $1\le a,b\le n$, $1\le i\le n-1$ and $i+1\le j\le n$.

\begin{Definition}
A {\bf $d$-chord} is a chord in an $(n+3)$-gon that belongs to the set $\mathcal{C}h_d$ defined as
\begin{eqnarray}\label{eq:def d-chords}
\left\{\overline{1i}:i\in[3,n+2]\right\}\cup\left\{\overline{j(n+3)}:j\in [2,n+1]\right\}\cup \left\{\overline{ij}:2\le i<j\le n+2,j-i\le d\right\}.
\end{eqnarray}
Two $d$-chords $\overline{kl},\overline{ab}\in\mathcal Ch_d$ are {\bf $d$-compatible}, denoted by $\overline{kl}\sim_d\overline{ab}$, if either they do not cross or they do cross and are of form
    \begin{eqnarray}\label{eq:allowed crossing}
    \overline{1j}\text{ and }\overline{i(n+3)}\text{ with }2\le i<j\le n+2\text{ and }j-i\ge {d+1}.
    \end{eqnarray}
The pair of chords in \eqref{eq:allowed crossing} is called a {\bf $k$-crossing} if $j-i=k$.
An {\bf $n$-chord diagram} is a collection of chords in an $(n+3)$-gon. 
We say that an $n$-chord diagram is {\bf $d$-admissible} if it is a maximal collection of {$d$-compatible} $d$-chords.
\end{Definition}

\begin{Remark}\label{rmk:crossing chords}
    The condition on crossing chords in \eqref{eq:allowed crossing} is equivalent to 
    \begin{eqnarray}\label{eq:allowed crossing Dyck}
    \overline{1j}\text{ and }\overline{i(n+3)}\text{ with }2\le i<j\le n+2\text{ and }\overline{ij}\not\in \mathcal Ch_d.
    \end{eqnarray}
    This description allows us to define a set of chords and accordingly chord diagrams associated to a Dyck path (instead of the parameter $d$). We make this precise in \S\ref{sec:Dyck}.
\end{Remark}

Observe that $\mathcal Ch_d\subset \mathcal Ch_{d+1}$ for all $d\le n-1$. 
Moreover, $\mathcal Ch_n$ is the set of all chords connecting non adjacent vertices in an $(n+3)$-gon and two chords are $n$-compatible precisely when they do not cross.
An example of two chord diagrams for $n=5$ is shown in Figure~\ref{fig:exp chord diag}.
Notice that a $(d-1)$-admissible chord diagram that does not contain any crossing chords of form $\overline{1j}$ and $\overline{i(n+3)}$ with $i<j$ and $j-i=d$ is also $d$-admissible.

The following Lemma translates compatibility from chords to rays in \eqref{eq:chords and rays}. It will be useful later on. It is a direct consequence from \eqref{eq:allowed crossing}.

\begin{Lemma}\label{lem:ray compatibility}
\begin{enumerate}
\item The following rays correspond to chords that do not cross $\overline{(i+1)(n+3)}$ corresponding to the ray $e_i$:
\begin{eqnarray}\label{eq:compatbile with e_i}
\begin{cases}
e_\ell & \ell \in[1,i-1]\cup [i+1,n],\\
e_\ell-e_k & (i\le \ell\le n-1,\, i+1\le k\le n \text{ or } 1\le \ell\le i-2,\, 2\le k\le i-1) \text{ and } k-\ell>1,\\
-e_k & k\in[1,i-1].
\end{cases}
\end{eqnarray}
Moreover, the rays $-e_k$ with $k\in[i+d,n]$ correspond to $(d-1)$-compatible crossings with $e_i$.
\item The following rays correspond to chords that do not cross $\overline{1(j+2)}$ corresponding to the ray $-e_j$:
\begin{eqnarray}\label{eq:compatbile with -e_j}
\begin{cases}
e_\ell & \ell \in [j+1,n],\\
e_\ell-e_k & (1\le \ell< k\le j \text{ or } j\le \ell<k\le n) \text{ and } k-\ell>1,\\
-e_k & k\in[1,j-1]\cup [j+1,n].
\end{cases}
\end{eqnarray}
Moreover, the rays $e_\ell$ with $\ell \in[1,j-d+1]$ correspond to $(d-1)$-compatible crossings with $-e_j$.
\item The following rays correspond to chords that do not cross $\overline{(i+1)(j+2)}$ corresponding to the ray $e_i-e_j$:
\[
\begin{cases}
e_\ell & \ell \in [1,i]\cup[j,n],\\
e_\ell-e_k & \ell,k+1\in[1,i]\cup[j+1,n+1]  \text{ and } k-\ell>1,\\
-e_k & k\in[1,i-1]\cup [j,n].
\end{cases}
\]
\end{enumerate}
\end{Lemma}

\begin{Definition}\label{def:lift chord diagrams}
Let $D$ be a $(d-1)$-admissible chord diagram that is not $d$-admissible.
So, $D$ contains a pair of crossing chords $\overline{1j}$ and $\overline{i(n+3)}$ with $i<j$ and $j-i=d$. Then define
\[
D_{ij}(1):=\left( D\setminus \left\{\overline{i(n+3)}\right\} \right)\cup\left\{\overline{ij}\right\} \quad\text{and}\quad D_{ij}(n+3):=\left(D\setminus 
\left\{\overline{1j}\right\}\right)\cup \left\{\overline{ij}\right\}
\]
We define the {\bf lift of $D$}, as the collection
\[
D^\uparrow:=\left\{D_{ij}(1),D_{ij}(n+3): \left\{\overline{1j},\overline{i(n+3)}\right\}\subset D \text{ with } i<j,j-i=d\right\}
\]
\end{Definition}
An example of the lift of a chord diagram is shown in Figure~\ref{fig:exp lift}.
By definition we have:
\begin{Lemma}
    The lift of a $(d-1)$-admissible chord diagram is a collection of $d$-admissible chord diagrams.
\end{Lemma}

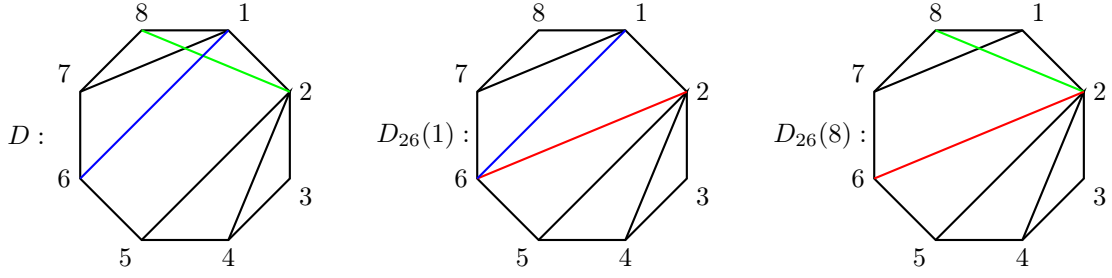
\begin{figure}
    \centering
\begin{tikzpicture}[scale=.75]

    \begin{scope}[xshift=-7cm]
 \coordinate (2) at (1.848, 0.765);   
    \coordinate (1) at (0.765, 1.848);   
    \coordinate (8) at (-0.765, 1.848);
    \coordinate (7) at (-1.848, 0.765);
    \coordinate (6) at (-1.848, -0.765);
    \coordinate (5) at (-0.765, -1.848);
    \coordinate (4) at (0.765, -1.848);
    \coordinate (3) at (1.848, -0.765);

    \draw[thick] (1) -- (2) -- (3) -- (4) -- (5) -- (6) -- (7) -- (8) -- cycle;
    \draw[thick] (7) -- (1);
    \draw[thick] (5)--(2)--(4);
\draw[green,thick] (2) -- (8);
\draw[blue,thick] (1) -- (6);

    \node[right] at (2) {2};
    \node[above right] at (1) {1};
    \node[above] at (8) {8};
    \node[above left] at (7) {7};
    \node[left] at (6) {6};
    \node[below left] at (5) {5};
    \node[below] at (4) {4};
    \node[below right] at (3) {3};

     \node at (-2.8,0) {$D:$};
    \end{scope}
 \coordinate (2) at (1.848, 0.765);   
    \coordinate (1) at (0.765, 1.848);   
    \coordinate (8) at (-0.765, 1.848);
    \coordinate (7) at (-1.848, 0.765);
    \coordinate (6) at (-1.848, -0.765);
    \coordinate (5) at (-0.765, -1.848);
    \coordinate (4) at (0.765, -1.848);
    \coordinate (3) at (1.848, -0.765);

    \draw[thick] (1) -- (2) -- (3) -- (4) -- (5) -- (6) -- (7) -- (8) -- cycle;

    \draw[thick] (7) -- (1);
    \draw[thick] (5)--(2)--(4);
\draw[red,thick] (2) -- (6);
\draw[blue,thick] (1) -- (6);

    \node[right] at (2) {2};
    \node[above right] at (1) {1};
    \node[above] at (8) {8};
    \node[above left] at (7) {7};
    \node[left] at (6) {6};
    \node[below left] at (5) {5};
    \node[below] at (4) {4};
    \node[below right] at (3) {3};

     \node at (-2.8,0) {$D_{26}(1):$};
    \begin{scope}[xshift=7cm]
 \coordinate (2) at (1.848, 0.765);   
    \coordinate (1) at (0.765, 1.848);   
    \coordinate (8) at (-0.765, 1.848);
    \coordinate (7) at (-1.848, 0.765);
    \coordinate (6) at (-1.848, -0.765);
    \coordinate (5) at (-0.765, -1.848);
    \coordinate (4) at (0.765, -1.848);
    \coordinate (3) at (1.848, -0.765);

    \draw[thick] (1) -- (2) -- (3) -- (4) -- (5) -- (6) -- (7) -- (8) -- cycle;
    \draw[thick] (7) -- (1);
    \draw[thick] (5)--(2)--(4);
\draw[green,thick] (2) -- (8);
\draw[red,thick] (2) -- (6);

    \node[right] at (2) {2};
    \node[above right] at (1) {1};
    \node[above] at (8) {8};
    \node[above left] at (7) {7};
    \node[left] at (6) {6};
    \node[below left] at (5) {5};
    \node[below] at (4) {4};
    \node[below right] at (3) {3};
         \node at (-2.8,0) {$D_{26}(8):$};
    \end{scope}
\end{tikzpicture}
    \caption{A $3$-admissible chord diagram $D$ on the left and its lift, the two $4$-admissible chord diagrams $D_{26}(1)$ and $D_{26}(8)$ in the center and on the right.}
    \label{fig:exp lift}
\end{figure}

\begin{Proposition}\label{prop:admissible}
     Let $D$ be a $d$-admissible chord diagram. Then
     \begin{enumerate}
         \item $D$ contains $n$ chords;
         \item if $D$ contains two $(d+1)$-crossings, then the chords of the crossings are pairwise distinct.
     \end{enumerate}
\end{Proposition}

\begin{proof}
\begin{enumerate}
\item If $D$ does not contain any crossings then it is a triangulation of an $(n+3)$-gon and therefore contains $n$ chords.
If $D$ does contain crossings, let $\Delta(D)$ be the set of $n$-admissible chord diagrams obtained as lifts of lifts of $D$.
Then every chord diagram $D'\in\Delta(D)$ has the same number of chords as $D$ and moreover it is a triangulation of the $(n+3)$-gon.
Hence, $D'$-- and therefore $D$-- has $n$ chords.
\item We need to verify that there are no $d$-admissible chord diagrams containing two distinct pairs of crossing chords $\overline{1j},\overline{i(n+3)}$ and $\overline{1j'},\overline{i'(n+3)}$ with $j-i=j'-i'=d+1$ that share a chord.
That is, we want to exclude $i=i'$ or $j=j'$ while $(i,j)\not=(i',j')$. 
But this cannot happen as $j-i=j'-i'=d+1$.
\end{enumerate}
\end{proof}

Each $d$-admissible chord diagram $D$ determines a cone $C_D$ in $\mathbb R^n$ spanned by the rays that are in bijection with chords in $D$ as given in \eqref{eq:chords and rays}.

\begin{Lemma}\label{lem:starsub chords}
Let $D$ be a $(d-1)$-admissible $n$-chord diagram that contains a pair of crossing chords $\overline{1j}$ and $\overline{i(n+3)}$ with $i<j$ and $j-i=d$. 
\begin{enumerate}
    \item The cones associated to the chord diagrams that form the lift of $D$, $C_{D_{ij}(1)}$ and $C_{D_{ij}(n+3)}$, form a fan.
    \item Let $\tau\subset C_D$ be the face generated by the rays corresponding to $\overline{1j}$ and $\overline{i(n+3)}$.  Then $\operatorname{Star}_\tau(C_D)= C_{D_{ij}(1)}\cup C_{D_{ij}(n+3)}$.
\end{enumerate}
\end{Lemma}
\begin{proof}
    \begin{enumerate}
        \item  Notice that by definition $C_{D_{ij}(1)}$ and $C_{D_{ij}(n+3)}$ share a facet, hence they form a fan.
        \item This is a direct consequence of Definition~\ref{def:lift chord diagrams} and the star subdivision.
    \end{enumerate}
\end{proof}

\begin{Definition}
Denote the union of all cones $C_D$ associated to $d$-admissible $n$-chord diagrams $D$  by $\mathcal C_{d,n}$. We call this set the {\bf chord diagram fan}.
\end{Definition}

The name is justified by our main result:

\begin{Theorem}\label{thm:chord fan}
Let $1\le d\le n$. Then the normal fan $\Sigma_{d,n}$ of $\mathcal P_{d,n}$ is the chord diagram fan $\mathcal C_{d,n}$.
More precisely, 
\begin{enumerate}
    \item A collection of rays spans a cone in $\Sigma_{d,n}$ if and only if the corresponding chords are pairwise $d$-compatible.
    \item Every maximal cone is generated by the $n$ rays corresponding to a $d$-admissible $n$-chord diagram.
\end{enumerate}
\end{Theorem}

\subsection{Proof of Theorem~\ref{thm:chord fan}}
The theorem relies on the fact that the refinements of $\Sigma_{d-1,n}$ by $\Sigma(Q_{d,i})$ are star subdivisions and we prove this by induction on $d$.
Define for $i\in[1,n-d+1]$
\begin{eqnarray}
    \tau_i=\cone(e_i,-e_{i+d-1}) \quad \text{and} \quad u_i=e_i-e_{i+d-1}.
\end{eqnarray}

Recall the ray generators of $\Sigma(Q_{d,i})$ from \eqref{eq:normal vectors simplex Qdi}. Label the vertices of $Q_{d,i}$ as $v_0=0, v_1=f_i, v_2= f_i+f_{i+1},\ldots,
v_{d}=f_i+\cdots+f_{i+d-1}.$ Moreover, denote the maximal cone in $\Sigma(Q_{d,i})$ corresponding to vertex $v_r$ by $\sigma_r$.
For a ray $\rho\in\Sigma_{d-1,n}(1)$, with primitive
generator $u_\rho$, define the minimizing set
\[
M_i(\rho)=\left\{r\in[0,d]:\langle v_r,u_\rho\rangle=\min_{s\in[0,d]}\{\langle v_s,u_\rho\rangle\}\right\}.
\]
In fact, the minimizing indices correspond precisely to the vertices of the $(n+3)$-gon that do not lie strictly on the boundary interval cut off by the chord corresponding to $\rho$.
This will become clear in the proof of the following Lemma:

\begin{Lemma}\label{lem:common-minimizer}
Let $A$ be a collection of rays whose chords  are pairwise $(d-1)$-compatible. Then
\[
\bigcap_{\rho\in A}M_i(\rho)=\varnothing \quad \Leftrightarrow \quad e_i,-e_{i+d-1}\in A.
\]
\end{Lemma}

\begin{proof}
For $0\leq r\leq d$ and $a\in[n]$, we have
\begin{eqnarray}\label{eq:pair with vertex}
    \langle v_r,e_a\rangle
=
\begin{cases}
1,&i\leq a\leq i+r-1,\\
0,&\text{otherwise}.
\end{cases}
\end{eqnarray}
We first compute the minimizing sets. 
If $a\notin[i,i+d-1]$, then $\langle v_r,e_a\rangle =0$ for all $r\in[0,d]$. 
So, $M_i(e_a)=M_i(-e_a)=[0,d]$. 
For $0\leq t\leq d-1$, we have $\min_{s\in[0,d]}\{\langle v_s,e_{i+t} \rangle\}=0$ and the minimum is achieved for $s\in[0,t]$.
Hence, $M_i(e_{i+s})=[0,s]$ and similarly $M_i(-e_{i+s})=[s+1,d]$.
For a ray $e_a-e_b$, where $a<b$ and $b-a\leq d-2$, the
nontrivial cases are
\[
M_i(e_a-e_b)=
\begin{cases}
[b-i+1,d],
   &a<i\leq b\leq i+d-1,\\
[0,a-i]\cup[b-i+1,d],
   &i\leq a<b\leq i+d-1,\\
[0,a-i],
   &i\leq a\leq i+d-1<b.
\end{cases}
\]
In all remaining cases, $M_i(e_a-e_b)=[0,d]$.
Equivalently, define the interval $I_i(\rho):=[0,d]\setminus M_i(\rho)$.
Notice that every nonempty $I_i(\rho)$ is an interval in $[0,d]$.
Then the above computation gives
\begin{eqnarray}\label{eq:cases I(rho)}
    I_i(\rho)=
    \begin{cases}
        [s+1,d], &  u_\rho=e_{i+s}, \\
        [0,s], &  u_\rho=-e_{i+s},\\
        [0,b-i], &  u_\rho =e_a-e_b,\ a<i\leq b\leq i+d-1, \\
        [a-i+1,b-i], &  u_\rho =e_a-e_b,\ i\leq a<b\leq i+d-1,\\
        [a-i+1,d], & u_\rho=e_a-e_b,\ i\leq a\leq i+d-1<b.
    \end{cases}
\end{eqnarray}
We use the convention $[r,s]=\varnothing$ if $s<r$.
Recall that by \eqref{eq:chords and rays} a ray corresponds to chord $\overline{a_\rho b_\rho}$ with $1\le a_\rho<b_\rho\le n+3$. 
It is not hard to verify that
\begin{eqnarray}\label{eq:Iirho from chords}
    I_i(\rho)= [a_\rho,b_\rho] \cap [i+1,i+d+1] - (i+1).
\end{eqnarray}
where $-(i+1)$ refers to a shift by $i+1$.
Then from \eqref{eq:cases I(rho)} we see that if a collection of $I_i(\rho)$, for pairwise $(d-1)$-compatible chords, covers all of $[0,d]$, then it must contain both $[1,d]=I_i(e_i)$ and $[0,d-1]=I_i(-e_{i+d-1})$.
{
Indeed, from \eqref{eq:cases I(rho)} we also deduce that any other minimal chain of intervals joining $0$ to $d$ would give two chords that cross in a way that is not $(d-1)$-compatible.
}
Consequently,
\[
\bigcup_{\rho\in A}I_i(\rho)=[0,d] \quad \Rightarrow \quad e_i,-e_{i+d-1}\in A.
\] 
Conversely, $M_i(e_i)=\{0\}$ and $M_i(-e_{i+d-1})=\{d\}$, so these two minimizing sets are disjoint.
\end{proof}

\begin{Proposition}\label{prop:stars}
Assume that $\Sigma_{d-1,n}=\mathcal C_{d-1,n}$.
Let $D$ be a $(d-1)$-admissible $n$-chord diagram.
Then 
\begin{enumerate}
\item If $\tau_i\not\subseteq C_D$, then there exists
$r\in[0,d]$ such that $C_D\subseteq \sigma_r$.
Consequently, $C_D$ is not subdivided in the common
refinement with $\Sigma(Q_{d,i})$.
\item If $\tau_i\subseteq C_D$, so that $C_D=\cone(e_i,-e_{i+d-1},\rho_1,\ldots,\rho_{n-2})$ for $\rho_1,\ldots,\rho_{n-2}$ rays.
Then
\begin{eqnarray*}
\begin{split}
C_D\cap \sigma_0 = \cone(e_i,u_i,\rho_1,\ldots,\rho_{n-2}),\quad \text{and} \quad C_D\cap \sigma_d = \cone(u_i,-e_{i+d-1},\rho_1,\ldots,\rho_{n-2}).
\end{split}
\end{eqnarray*}
All intersections of $C_D$ with the remaining cones
$\sigma_1,\ldots,\sigma_{d-1}$ are faces of these two cones.
\item  The common refinement of $\Sigma_{d-1,n}$ and $\Sigma(Q_{d,i})$, denoted $\Sigma_{d-1,n}\wedge \Sigma(Q_{d,i})$, is the star subdivision of $\Sigma_{d-1,n}$ at $\tau_i$.
\end{enumerate}
\end{Proposition}

\begin{proof}
The maximal cones of $\Sigma(Q_{d,i})$, modulo lineality space, are
\[
\sigma_r=\{w:\langle v_r,w\rangle \le \langle v_t,w\rangle
\text{ for all }0\le t\le d\},
\qquad 0\le r\le d.
\]
That is, $\rho$ lies in $\sigma_r$ if and only if  $r \in M_i(\rho)$.
By Lemma~\ref{lem:common-minimizer} the intersection of all $M_i(\rho)$ for $\rho\in C_D(1)$ is non empty, unless $e_i,-e_{i+d-1}\in C_D(1)$.
The latter pair is the unique compatible pair for which
the two sets of minimizing indices are disjoint. 
So, unless $C_D$ contains these two rays, we can find an $r$ so that $C_D$ lies in $\sigma_r$.

For (2), suppose that $e_i,-e_{i+d-1}\in C_D(1)$.
Every other ray $\rho\in C_D(1)$ is compatible with both of these rays. 
Lemma~\ref{lem:ray compatibility} lists the possible cases for $\rho$. 
In particular, compatibility with $e_i$ and $-e_{i+d-1}$ implies that the ray generator $u_\rho$ lies in the intersection of the sets defined in \eqref{eq:compatbile with e_i} and \eqref{eq:compatbile with -e_j}.
It is therefore supported on $[1,i-1]\cup[i+d,n]$.
Combining with \eqref{eq:pair with vertex} we deduce $\langle v_0,u_\rho\rangle=\langle v_d,u_\rho\rangle=0$ and $\langle v_r,u_\rho\rangle \ge0$ for $1\le r\le d-1$.
Now take an arbitrary point $w\in C_D$. 
It has an expression of form $w = a e_i-be_{i+d-1} +\sum_{\ell=1}^{n-2}c_\ell u_{\rho_\ell}$ where $a,b,c_\ell\ge0$.
In particular, $\langle v_0,w\rangle=0$ and $\langle v_d,w\rangle=a-b$
while 
\[
\langle v_r,w\rangle
=
a+\sum_{\ell=1}^{n-2} c_\ell \langle v_r,u_{\rho_\ell}\rangle \ge a
\quad\text{for }1\le r\le d-1.
\]
Thus the minimum is attained at $r=0$ when $a\ge b$, and at
$r=d$ when $b\ge a$. Therefore
\[
C_D\cap \sigma_0 = \{w\in C_D:a\ge b\}
\quad \text{and} \quad C_D\cap \sigma_d = \{w\in C_D:b\ge a\}.
\]
If $a\ge b$, then $a e_i+b(-e_{i+d-1}) = (a-b)e_i+b(e_i-e_{i+d-1})$,
so $C_D\cap \sigma_0= \cone(e_i,u_i,\rho_1,\ldots,\rho_{n-2})$.
Similarly for $b\ge a$ we have $a e_i+b(-e_{i+d-1}) = a(e_i-e_{i+d-1})
+(b-a)(-e_{i+d-1})$, and hence $C_D\cap \sigma_d = \cone(u_i,-e_{i+d-1}, \rho_1,\ldots,\rho_{n-2})$. This implies (2).

For (3), notice that the two cones $C_D\cap \sigma_d$ and $C_D\cap \sigma_0$ are precisely the two cones obtained by subdividing
$C_D$ along $\tau_i$.
Moreover, by (1) every maximal cone not containing $\tau_i$ is contained in a cone of $\Sigma(Q_{d,i})$ and is therefore unchanged in the common refinement. The same is true in the star subdivision along $\tau_i$.
This implies (3).
\end{proof}

We are now prepared to prove the main result, Theorem~\ref{thm:chord fan}.

\begin{proof}[Proof of Theorem~\ref{thm:chord fan}]
We proceed by induction on $d$ to show that $\Sigma_{d,n}=\mathcal C_{d,n}$. 
    
\medskip
\noindent
    \underline{Induction start:} $\mathcal P_{1,n}$ is a hypercube with rays $\pm e_a$ for $a\in[n]$. Any two rays appear together as generators in a cone except $e_a$ and $-e_a$.
    By \eqref{eq:chords and rays} we obtain the collection of chords $\overline{1(a+2)}$ and $\overline{(a+1)(n+3)}$, $a\in[n]$. 
    The maximal cones of $\Sigma(\mathcal P_{1,n})$ induce chord diagrams with the aforementioned chords in which all chords may cross except $\overline{1(a+2)}$ and $\overline{(a+1)(n+3)}$. 
    This translates the fact that only pairs of chords of form $\overline{1j}$ and $\overline{i(n+3)}$ with $2\le i<j\le n+2$ and $j-i\ge 2$ may cross.
    Hence, every maximal cone in $\Sigma(\mathcal P_{1,n})$ corresponds to a $1$-admissible $n$-chord diagram.

\medskip
\noindent
    \underline{Induction step:} Assume that the maximal cones of $\Sigma_{d-1,n}$ are in bijection with $(d-1)$-admissible $n$-chord diagrams. By definition, $\Sigma_{d,n}$ is obtained from $\Sigma_{d-1,n}$ as the refinement 
    \[
     \Sigma_{d-1,n} \wedge \bigwedge_{i=1}^{n-d+1} \Sigma(Q_{d,i}).
    \]
    We consider the refinement step by step, that is for each $Q_{d,i}$ separately.
    By Lemma~\ref{lem:rays Q}, $\Sigma(Q_{d,i})$ is the normal fan of a $d$-dimensional simplex in $\mathbb R^n$. 
    Its rays are
    \[
    -e_i,e_i-e_{i+1},\dots,e_{i+d-2}-e_{i+d-1},e_{i+d-1}
    \]
    which are in bijection with chords:
    \[
    \overline{1(i+2)},\overline{(i+1)(i+3)},\overline{(i+2)(i+4)},\dots,\overline{(i+d-1)(i+d+1)},\overline{(i+d)(n+3)}.
    \]
By Proposition~\ref{prop:stars} it suffices to show that the refinements $\Sigma_{d-1,n}\wedge \Sigma(Q_{d,i})$ and  $\Sigma_{d-1,n}\wedge \Sigma(Q_{d,j})$ for $i\not=j$ correspond to commuting star subdivisions of $\Sigma_{d-1,n}$.
By Lemma~\ref{lem:commuting stars} it suffices to check that if a cone $\sigma\in\Sigma$ contains two cones $\tau_i$ and $\tau_j$, then $\tau_i(1)\cap \tau_j(1)=\varnothing$.
Proposition~\ref{prop:admissible}(2) shows that if a $(d-1)$-admissible $n$-chord diagram contains two $d$-crossings, then the chords involved are pairwise distinct. 
That is, if a cone $\sigma\in\Sigma_{d-1,n}$ contains two cones corresponding to $(d-1)$-crossings, say $\tau_i$ and $\tau_j$, then their ray sets are disjoint: $\tau_i(1)\cap \tau_j(1)=\varnothing$.
Hence, by Lemma~\ref{lem:commuting stars}
\[
\Star_{\tau_i}\left(\Star_{\tau_j}(\sigma)\right) = \Star_{\tau_j}\left(\Star_{\tau_i}(\sigma)\right).
\]
So the star subdivisions can be performed successively, in arbitrary order, with each star subdivision introducing exactly the one ray of form $e_i-e_{i+d-1}$.
Moreover, iterating the refinement and using the claim for all $i\in[n-d+1]$ we obtain $\Sigma_{d,n}$ as a sequence of star subdivisions at all two dimensional cones corresponding to $d$-crossings.
A cone $C_D$ in $\Sigma_{d-1,n}$ of a chord diagram $D$ containing the $d$-crossing  $\overline{1j}$ and $\overline{i(n+3)}$ is replaced by $C_{D_{ij}(1)}\cup C_{D_{ij}(n+3)}$.
By Lemma~\ref{lem:starsub chords}, these two cones are exactly the corresponding star subdivision of $C_D$.
Hence, every maximal cone in $\Sigma_{d,n}$ corresponds to a $d$-admissible chord diagram. 
This simultaneously concludes the proof of (1) and (2).
\end{proof}

An immediate consequence is the following: 

\begin{Corollary}\label{cor:rays}
    The normal fan of the Catalan pellytope $\Sigma_{d,n}$ has $n+\sum_{j=n-d+1}^n j$ rays generated by
\begin{equation}\label{eq:rays}
    \pm e_i \text{ for }i\in [n] \quad \text{and} \quad e_i-e_{i+k} \text{ for } i\in[n-k],k\in [d-1].
\end{equation}
\end{Corollary}

\subsection{Properties of chord diagram fans}

In this section, we explore combinatorial and geometric properties of chord diagram fans. 
In particular, we show that $\Sigma_{d,n}$ is obtained from $\Sigma_{d-1,n}$ by an explicit sequence of star subdivisions induced from the chord diagrams.
Moreover, we show that the family of fans $\{\Sigma_{d,n}:n\in\mathbb N, d\le n\}$ is recursively closed under links.

\begin{Lemma}\label{lem:prep flag}
For a collection $D\subset \mathcal Ch_d$, the following are equivalent:
\begin{enumerate}
    \item the chords in $D$ are pairwise $d$-compatible;
    \item $D$ is contained in a $d$-admissible diagram;
    \item the corresponding rays generate a cone of $\Sigma_{d,n}$.
\end{enumerate}
\end{Lemma}
\begin{proof}
By definition, a $d$-admissible diagram is a maximal collection of $d$-compatible chords. Moreover, by Proposition~\ref{prop:admissible}(1), all $d$-admissible diagrams have $n$ chords. 
So, (1) implies (2).
Notice further that (2) implies (1) by definition.

To see that (1) implies (3) let $\{D_1,\dots,D_s\}$ be the set of all maximal collections of $d$-compatible chords containing the rays corresponding to chords in $D$. 
By definition, $D_1,\dots,D_s$ are $d$-admissible.
Moreover, $C_{D_1},\dots,C_{D_s}$ are cones in the fan $\Sigma_{d,n}=\mathcal C_{d,n}$ by Theorem~\ref{thm:chord fan}.
The common intersection of $C_{D_1},\dots,C_{D_s}$ is again a cone in $\Sigma_{d,n}$.
By construction, this cone is generated by the rays corresponding to $D$. 

Lastly, (3) implies (1) by Theorem~\ref{thm:chord fan}.
\end{proof}

Recall the definition of a flag complex from \S\ref{sec:pre}.

The following result for the special case $d=2$ can be found in \cite[Lemma 3.3]{Pellytopes}. The ideas of this result have already appeared in the proof of Theorem~\ref{thm:chord fan}.

\begin{Proposition}\label{prop: star of chord fan}
The fan $\Sigma_{d,n}$ is obtained from $\Sigma_{d-1,n}$ by a successive sequence of star subdivisions that correspond to lifting all $d$-crossings in  $(d-1)$-admissible $n$-chord diagrams.
\end{Proposition}

\begin{proof}
Notice that $\Sigma_{d-1,n}=\mathcal C_{d-1,n}$ is simplicial. 
So we have by Proposition~\ref{prop:stars} that $\Sigma_{d-1,n}\wedge \Sigma(Q_{d,i})=\operatorname{Star}_{\tau_i}(\Sigma_{d-1,n})$.
We need to show that the successive star subdivisions occur in disjoint subfans.
As we have seen in the proof of Proposition~\ref{prop:stars} a maximal cone $\sigma\in\Sigma_{d-1,n}$ is subdivided in the refinement with $\Sigma(Q_{d,i})$ if and only if it contains $\tau_i=\cone(e_i,-e_{i+d-1})$.
This corresponds to a chord diagram with crossing chords $\overline{(i+1)(n+3)}$ and $\overline{1(i+d+1)}$.
Moreover, the new ray $e_i-e_{i+d-1}$ corresponds to the chord $\overline{(i+1)(i+d+1)}$.
Therefore, the refinement corresponds exactly to the lift of the chord diagrams.
Recall from Proposition~\ref{prop:admissible}(2): if a $(d-1)$-admissible $n$-chord diagram contains $k$ pairs of crossing chords $\overline{1j_p},\overline{i_p(n+3)}$ with $2\le i_p<j_p\le n+2$ and $j_p-i_p=d$ for $p=1,\dots,k$, then all chords involved in these crossings are pairwise distinct. 
Hence, star subdivisions at the corresponding cones $\tau_{i_p}$ pairwise commute by Lemma~\ref{lem:commuting stars}. 
We conclude that
\begin{eqnarray}\label{eq:succesive stars}
    \Sigma_{d,n}=\operatorname{Star}_{\tau_1}(\operatorname{Star}_{\tau_2}(\cdots \operatorname{Star}_{\tau_{n-d+1}}(\Sigma_{d-1,n})\cdots)).
\end{eqnarray}
\end{proof}

\begin{Corollary}\label{cor:smooth flag}
$\Sigma_{d,n}$ is smooth, in particular simplicial, and flag.
\end{Corollary}
\begin{proof}
By Proposition~\ref{prop:stars}, every cone $C_D$ is obtained as a sequence of star subdivisions along cones of form $\tau_i$ from a maximal cone in $\Sigma_{1,n}$.
As all maximal cones in $\Sigma_{1,n}$ are smooth, and also all $\tau_i$ are smooth, so is $C_D$.
Therefore $\Sigma_{d,n}$ is smooth.

It is left to show that $\Sigma_{d,n}$ is flag: that is, for any given set of rays $\{\rho_1,\dots,\rho_r\}$ with the property that every $\{\rho_i,\rho_j\}$ with $1\le i<j\le r$ generates a cone in $\Sigma_{d,n}$, then also $\{\rho_1,\dots,\rho_r\}$ generates a cone in $\Sigma_{d,n}$. 
The collection of rays $\{\rho_1,\dots,\rho_r\}$ corresponds to a collection $\sigma$ of $d$-compatible chords in the $(n+3)$-gon as $\Sigma_{d,n}=\mathcal C_{d,n}$.
By Lemma~\ref{lem:prep flag} $\{\rho_1,\dots,\rho_r\}$ generate a cone.
\end{proof}

From now on we identify rays of $\Sigma_{d,n}$ with chords in $\mathcal Ch_d$ to simplify notation.
Lastly in this section we explore the local structure of the chord diagram fan. 
In fact, the recursive decomposition of the associahedron along a diagonal persists throughout the entire Catalan pellytope family. 
Recall the notion of link from \S\ref{sec:pre} and, for convention, set $\Sigma_{0,0}=\{\{0\}\}$.

\begin{Proposition}\label{prop:link}
    Let $\overline{ij}\in\mathcal Ch_d$. Then $ \link_{\Sigma_{d,n}}(\overline{ij})=\Sigma_{d_1,n_1}\times \Sigma_{d_2,n_2}$ with $n_1,n_2\ge 0$ such that $n_1+n_2=n-1$ and for each $i\in\{1,2\}$ either $(d_i,n_i)=(0,0)$ or $d_i\in[n_i]$ with $1\le d_1+d_2\le n-1$.
\end{Proposition}

The precise values of $n_1,n_2,d_1,d_2$ depend on the chord $\overline{ij}$ and are given explicitly in the proof.

\begin{proof}
The chord $\overline{ij}$ cuts the $(n+3)$-gon into two subpolygons: an $(n_1+3)$-gon and an $(n_2+3)$-gon.
We need to show that any $d$-admissible $n$-chord diagram in the $(n+3)$-gon containing $\overline{ij}$ induces $d_k$-admissible chord diagrams in the subpolygons for $k\in\{1,2\}$.
We distinguish three cases for $\overline{ij}$ and analyze which $d$-admissible chord diagram contains this chord. 

\begin{enumerate}
\item \underline{$i>1,j<n+3$:} In this case no chords may cross $\overline{ij}$. 
The chord $\overline{ij}$ splits the $(n+3)$-gon into two polygons, one with vertices $1,\dots,i,j,\dots,n+3$ so $n_1+3=n+3-j+i+1$ and one with vertices $i,i+1,\dots,j-1,j$ so that $n_2+3=j-i+1$.
In the first case, $d$ is diminished by the distance of $i$ and $j$, that is $d_1=d-(j-i)+1$.
In the second case, we don't have any crossing chords as the subpolygon does not contain the original vertices $1$ and $n+3$. 
Hence, $d_2=n_2$. 
In particular, $n_1+n_2=n-1$ and $d_1+d_2=d-1$.

\item \underline{$i=1$:} In this case the chord $\overline{1j}$ splits the $(n+3)$-gon into a $j$-gon with vertices $1,\dots,j$ and a polygon with vertices $j,\dots,n+3,1$.
In particular, $n_1=j-3$ and $n_2=n-j+2$ so that $n_1+n_2=n-1$.
Consider a chord $\overline{i'(n+3)}$ with $i'<j$ that crosses $\overline{1j}$.
We identify it with the chord $\overline{i'j}$ in the $j$-gon. 
This way any $d$-admissible chord diagram containing such a crossing induces a triangulation of the $(n_1+3) $-subpolygon. 
Hence, there are no crossings in the induced $(n_1+3)$-gon so that $d_1=n_1$.

In the $(n_2+3)$-gon we might have crossings that do not involve the chord $\overline{1j}$ as this subpolygon contains the vertices $1$ and $n+3$. This implies that the induced chord diagrams are $d$-admissible, so $d_2=\min\{d,n_2\}$.  
\item \underline{$j=n+3$:} This chord splits the $(n+3)$-gon into a subpolygon with vertices $i,i+1,\dots,n+3$ and a second one with vertices $1,\dots,i,n+3$.
So we have $n_1=n-i+1$ and $n_2=i-2$. 
Whenever a $d$-admissible chord diagram that contains $\overline{i(n+3)}$ further contains crossing chords $\overline{1j'}$ and $\overline{i'(n+3)}$ the indices satisfy $j'>i'\ge i$. 
We replace the chord $\overline{1j'}$ by $\overline{i'j'}$ so that the $n_1+3$-gon does not contain any crossing chords. In particular, $d_1=n_1$.

As we have already assigned chords of form $\overline{1j'}$ to chords $\overline{i'j'}$ for $j'>i'\ge i$ the $n_2+3$-subpolygon may only contain crossings of form $\overline{1\ell}$ and $\overline{k(n+3)}$ with $k<\ell<i$. This implies that $d_2=\min\{d,n_2\}$.
Hence, $n_1+n_2=n-1$ and $d_1+d_2\le n-1$. 
\end{enumerate}
\end{proof}

\subsection{Number sequences counting vertices in Catalan pellytopes}\label{sec:numbers}
The aim of this section is to count the number of vertices of Catalan pellytopes. As this family of polytopes contains a few well-known classes of polytopes-- pellytopes and associahedra-- we start by recalling the number sequences associated to their vertex counts.

Pell numbers are defined through a three-term recurrence relation similar to Fibonacci numbers.
The {\bf $n^{\text{th}}$ Pell number} is defined through the recurrence relation
  \begin{equation}\label{eq:pellrec}
P_0=0, \quad P_1=1,\quad \text{and } \quad P_n=2P_{n-1}+P_{n-2}. 
\end{equation}
Pell numbers appear in the Online Encyclopedia for Integer Sequences as \cite[A000129]{OEIS}.
Some of the first Pell numbers are displayed in the second column of Table~\ref{tab:S_n^d}.
By \cite{Pellytopes} Pell numbers count the number of vertices of the pellytope $\mathcal P_{2,n}$.
The recurrence relation for Pell numbers can be solved yielding the closed expression
\begin{equation*}
    P_n = \frac{(1+\sqrt{2})^n-(1-\sqrt{2})^n}{2\sqrt{2}}\,.
\end{equation*}
Moreover, a generating function for Pell numbers $P(x) := \sum_{n=0}^\infty P_n x^n$ is known
$P(x) =  \frac{1}{1-2x-x^2}$.  

We now move to the associahedron, $\mathcal P_{n,n}$, whose number of vertices is counted by Catalan numbers. 
The {\bf $n^{\text{th}}$ Catalan number} is defined as
  \begin{equation}\label{eq:def Catalan}
    C_n
    =
    \frac{1}{1+n} {2n \choose n}\,.
  \end{equation}
  As a convention set $C_{-1}:=-1$.
Catalan numbers satisfy a quadratic $(n+1)$-term recurrence relation
\begin{equation*}
  C_n =  \sum_{k=1}^n C_{k-1}C_{n-k}
\end{equation*}
The generating function for Catalan numbers $C(x) := \sum_{n=0}^\infty C_n x^n$ is given by $C(x) = \frac{1-\sqrt{1-4x}}{2x}$.

As the Catalan pellytopes interpolate between the pellytope and the associahedron we obtain number sequences interpolating between Pell and Catalan numbers. 
Define $S^{(d)}_n$ as the number of $d$-admissible chord diagrams in an $(n+3)$-gon.

\begin{Example}\label{exp:number sequences}
We already know three special cases:
\begin{enumerate}
    \item[d=1] As $\mathcal P_{1,n}$ is the hypercube we have $S_n^{(1)}=2^n$. 
    \item[d=2] The polytope $\mathcal P_{2,n}$ is the pellytope, so $S_n^{(2)}=P_{n+1}$ by \cite[Corollary 3.2]{Pellytopes}. 
    \item[d=n] As $\mathcal P_{n,n}$ is the associahedron, we have $S_{n}^{(n)}=C_{n+1}$ for all $n$.
\end{enumerate}
\end{Example}

For $d=3$ the sequence $S_n^{(3)}$ coincides with \cite[A077938]{OEIS}, defined as the expansion of $\frac{1}{1 - 2x - x^2 - 2x^3}$.
Corollary~\ref{cor:generating function} below shows why this is the case. 
For $d\ge 4$ there are no entries in OEIS \cite{OEIS} that coincide with $S_n^{(d)}$.
Interestingly, it is mentioned in \cite[A077938]{OEIS} that this sequence can be viewed as the second member of an infinite family of sequences, with Pell numbers as the first, where subsequent sequences have longer recurrences. 
However, these do not coincide with our sequences.

\begin{Proposition}\label{prop:recursion}
The numbers $S_n^{(d)}$ for $n\in\mathbb N$ and $d\le n$ satisfy the following recurrence relation
\begin{equation}\label{eq:recursion}
  S^{(d)}_n = 2 S^{(d)}_{n-1} + \sum_{k=1}^{d-1} C_k S^{(d)}_{n-k-1}
\end{equation}
\end{Proposition}

The result follows from the recursive patterns satisfied by chord diagrams that we already used in the proof of Proposition~\ref{prop:link}. 
More precisely, to prove the recursion we will count the number of cases of type (1) and (2) in the proof of Proposition~\ref{prop:link}.

\begin{proof}[Proof of Proposition~\ref{prop:recursion}]
The cases $d=1$ and $d=2$ follow from Example~\ref{exp:number sequences}. For $d=1$ we have $S_n^{(1)}=2^n=2S_{n-1}^{(1)}$.
For $d=2$ using \eqref{eq:pellrec} we have
\[
S_n^{(2)}=P_{n+1}=2P_{n}+P_{n-1}=2S_{n-1}^{(2)}+C_1S_{n-2}^{(2)}.
\]
So we may assume $d\ge 3$. Let $D$ be a $d$-admissible $n$-chord diagram.
Observe that $\overline{13}$ and $\overline{2(n+3)}$ are not $d$-compatible. This yields three cases: (1) $\overline{13}\in D$; (2) $\overline{2(n+3)}\in D$; or (3) $\overline{13},\overline{2(n+3)}\not\in D$.

\medskip\noindent
{\it Case (1):} Assume $\overline{13}\in D$. As $\overline{2(n+3)}\not\in D$ there is no chord in $D$ crossing $\overline{13}$.
We cut off the triangle with vertices $1,2,3$ and get an $(n+2)$-gon with vertices $1,3,4,\dots,n+3$. After relabeling to $1,2,\dots,n+2$ we are left with a $d$-admissible $(n-1)$-chord diagram. Hence, 
\[
\vert\{D\in\mathcal C_{d,n}:\overline{13}\in D\}\vert=S_{n-1}^{(d)}.
\]

\medskip\noindent
{\it Case (2):} Assume $\overline{2(n+3)}\in D$. 
This is a special case of Case (2) in the proof of Proposition~\ref{prop:link}. We claim there is a bijection
\[
\{D\in\mathcal C_{d,n}:\overline{2(n+3)}\in D\} \leftrightarrow \{D'\in\mathcal C_{d,n-1}\}
\]
Given $D$, we construct $D'$ by contracting the $\overline{12}$ boundary segment. Relabel the vertices by $1,2,\dots,n+2$.
If $D$ contains crossing chords $\overline{i(n+3)}$ and $\overline{1j}$ they correspond to crossing chords $\overline{(i-1)(n+2)}$ and $\overline{1(j-1)}$ in $D'$ which remain $d$-compatible.
For the reverse, given $D'\in\mathcal C_{d,n-1}$ add a triangle to the boundary segment $\overline{1(n+2)}$ and relabel accordingly.
We deduce,
\[
\vert\{D\in\mathcal C_{d,n}:\overline{2(n+3)}\in D\}\vert=S_{n-1}^{(d)}.
\]
{\it Cases} (1) and (2) together account for the term $2S_{n-1}^{(d)}$ in the recursion \eqref{eq:recursion}.

\medskip\noindent
{\it Case (3):} Assume neither $\overline{13}$ nor $\overline{2(n+3)}$ are contained in $D$. 
Then $D$ must contain a chord of form $\overline{1j}$ as otherwise we may add $\overline{2(n+3)}$.
Let $\overline{1j}\in D$ be such that $j$ is minimal among all $\ell$ with $\overline{1\ell}\in D$.
As $\overline{13}\not \in D$ we have $j\ge 4$.
We claim $j\le d+2$.
Assume to the contrary that $j\ge d+3$. Then every chord $\overline{1\ell}\in D$ satisfies $\ell-2\ge j-2\ge d+1$.
Hence, all chords of form $\overline{1\ell}\in D$ are compatible with $\overline{2(n+3)}$, so we can add it to $D$. However, this contradicts the maximality of $D$.
Therefore, $4\le j\le d+2$. 

Set $k=j-3$ so that $1\le k\le d-1$. 
This is the index of the sum in the recursion \eqref{eq:recursion}.
We now proceed as in Case (2) in the proof of Proposition~\ref{prop:link}.
Consider the subpolygon with vertices $1,2,\dots,j$. As it does not contain the vertex $n+3$ there are no crossings inside this polygon.
Therefore the subpolygon is triangulated.
Moreover, it does not contain any of the chords $\overline{13},\overline{14},\dots,\overline{1(j-1)}$.
Therefore, it must contain the chord $\overline{2(j-1)}$.
Hence, it is equivalent to a triangulation of a $(k+2)$-gon with vertices $2,3,\dots,j$. So there are $C_k$ possibilities. 

What is left to show is how to recover the factor $S_{n-k-1}^{(d)}$. Consider the subpolygon with vertices $1,j,j+1,\dots,n+3$. It has $1+(n+3-j+1)=n-j+5=n-k+2$ vertices. 
As $k=j-3$ it is therefore an $(n-k+3)$-gon. 
Note that as in {\it Case (2)}, $d$-compatibility is preserved under relabeling of the vertices and hence there are $S_{n-k-1}^{(d)}$ possibilities.

Summarizing all three cases we obtain the desired recursion formula.
\end{proof}

We use \eqref{eq:recursion} to give a generating function for $S_{n}^{(d)}$. For this purpose, set $S_0^{(d)}=1$ and $S_n^{(d)}=0$ for $n<0$.

\begin{Corollary}\label{cor:generating function}
The numbers $S^{(d)}_n$ are determined by the generating function $S^{(d)}(x) = \sum_{n=0}^\infty S^{(d)}_nx^n$:
\begin{equation}\label{eq:genfunc}
S^{(d)}(x) = \frac{1}{1-2x- \sum_{k=1}^{d-1}C_{k} x^{k+1}}.
\end{equation}
\end{Corollary}

\begin{proof}
From \eqref{eq:recursion} we obtain 
\begin{eqnarray}\label{eq:sum expansion}
     \sum_{n \ge 0} S^{(d)}_n x^n = 2\sum_{n \ge 0} S^{(d)}_{n-1} x^n + \sum_{ n\ge 0}\sum_{k=1}^{d-1} C_k S^{(d)}_{n-k-1} x^n
\end{eqnarray}
where the left hand side satisfies
$\sum_{n\ge 0} S^{(d)}_n x^n = S^{(d)}(x)-S_0^{(d)} = S^{(d)}(x) -1$. 
On the right hand side we have $2\sum_{n \ge 0} S^{(d)}_{n-1} x^n  = 2 x S^{(d)}(x)$.
The double sum moreover is 
\[
\sum_{ n\ge 0}\sum_{k=1}^{d-1} C_k S^{(d)}_{n-k-1} x^n = \sum_{k=1}^{d-1} C_k \sum_{n\ge k+1} S^{(d)}_{n-k-1} x^n=  \sum_{k=1}^{d-1} C_k x^{k+1}\sum_{m\ge 0} S_m^{(d)} x^m = \sum_{k=1}^{d-1} C_k x^{k+1} S^{(d)}(x).
\]
Combining the above, from \eqref{eq:sum expansion} we obtain
\[
S^{(d)}(x)\left(1-2 x - \sum_{k=1}^{d-1} C_k x^{k+1} \right) = 1
\]
and so the claim follows. 
\end{proof}

\begin{Remark}
    The generating function $S^{(d)}$ has an alternative expression as
    \[
   S^{(d)} = \left[\frac{\binom{2 d}{d} x^{d+1} \, _2F_1\left(1,d+\frac{1}{2};d+2;4 x\right)}{d+1}-x+\frac{1}{2} \sqrt{1-4 x}+\frac{1}{2}
    \right]^{-1}
    \]
where $_2F_1(a,b,c;x) = \sum_{n=0}^\infty \frac{(a)_n(b)_n}{(c)_n} \frac{x^n}{n!}$ is the Gau\ss{}  hypergeometric function. 
\end{Remark}

Recall that the Catalan pellytope $\mathcal P_{d,n}$ is defined for $d\le n$ and $n\ge 1$. 
Hence, the numbers $S_n^{(d)}$ are defined for $d\le n$ and $n\ge 1$. 
The recursion formula \eqref{eq:recursion} however suggests a natural extension beyond the case $d\le n$. 
Consider the adapted recursion relation
\begin{equation}\label{eq:recursion incl catalan}
  \mathcal S^{(d)}_n = 2 \mathcal S^{(d)}_{n-1} + \sum_{k=1}^{\min\{d-1,n-1\}} C_k \mathcal S^{(d)}_{n-k-1}
\end{equation}

\begin{Lemma}
Let $(\mathcal S_n^{(d)})_{n\ge 1}$ be a sequence of numbers satisfying \eqref{eq:recursion incl catalan}. Then for $n<d$ we have $\mathcal S^{(d)}_n=C_{n+1}$.
\end{Lemma}
\begin{proof}
By induction, we have
\begin{equation*}
  \mathcal S^{(d)}_n = 2 \mathcal S^{(d)}_{n-1} + \sum_{k=1}^{n-1} C_k \mathcal S^{(d)}_{n-k-1} =  2 C_{n} + \sum_{k=1}^{n-1} C_k C_{n-k} =\sum_{k=0}^n C_kC_{n-k}=C_{n+1}.
\end{equation*}
\end{proof}

Hence, to define a number sequence satisfying \eqref{eq:recursion incl catalan} we only need to add the first $d$ entries as Catalan numbers.

\begin{Definition}
The {\bf $n^{\text{th}}$ $d$-Catalan Pell number} $\mathcal S_n^{(d)}$ for arbitrary $d\ge 1$ is defined by the recurrence relation \eqref{eq:recursion incl catalan} subject to $d$ initial conditions $\mathcal S_n^{(d)}=C_{n+1}$ for $n\le d$.    
\end{Definition}

Summarizing, Catalan and Pell numbers appear as the following special cases of Catalan Pell numbers:
\begin{eqnarray}\label{eq:summary numbers}
\mathcal S_n^{(1)}=2^n,\quad \mathcal S_n^{(2)}=P_{n+1}, \quad \mathcal S_n^{(d)}=C_{n+1} \text{ for } n\le d.
\end{eqnarray}
Moreover, $\lim_{d\to \infty} \mathcal S_n^{(d)}= C_{n+1}$. Geometrically, for fixed $n$, sequences of form $(\mathcal S_n^1,\, \mathcal S_n^2,\,\dots, \,\mathcal S_n^n)$ are interesting. 
By Proposition~\ref{prop: star of chord fan} the difference $\mathcal S^{(d)}_n -\mathcal S^{(d-1)}_{n}$ is the number of star subdivisions applied to $\Sigma_{d-1,n}$ in order to obtain $\Sigma_{d,n}$. Equivalently, it is the number of $(d-1)$-admissible $n$-chord diagrams in that contain $d$-crossings (counted with multiplicity).

\subsection{Dyck paths and chamber Ansatz}\label{sec:Dyck}

In this section we explain the connection between chord diagram fans and Dyck paths. 
This identifies Catalan pellytopes as special cases of polytopes studied in \cite{CalvoCortesFrost2026}.
In order to do so, consider the {\bf directed graph $\mathcal D_n$} that has 
\begin{enumerate}
    \item vertices labelled by $(k,i)$ where $k\in[0,n]$ and $i\in[n-k+1]$.
    \item arrows from the vertex $(k,i)$ for $k\in[0,n],i\in[n-k+1]$ to the vertices $(k
    +1,i),(k-1,i+1)$, whenever these vertices exist.
\end{enumerate}
By a little abuse of notation, we label the vertices by $q_{k,i}$ for $k\in[n],i\in[n-k+1]$ and by $1$ otherwise. 
We call a cycle consisting of four edges in the underlying non oriented graph a {\bf face of $\mathcal D_n$}.
Then label the face of $\mathcal D$ bound by the vertices $(k+1,i),(k,i),(k-1,i+1),(k,i+1)$ by $u_{i+1,k+i+2}$.
Additionally, we place the vertex $(0,0)$ and the vertex $(0,n+1)$ in a box.
The diagram $\mathcal D_n$ is depicted in Figure~\ref{fig:dyck diagram}.

A {\bf Dyck path} in $\mathcal D_n$ is an oriented path from vertex $(0,0)$ to vertex $(0,n+1)$ that does not contain any other vertices $(0,a)$.
Notice that every Dyck path consists of exactly $n$ {\bf upward arrows} of form $(k,i)\to (k+1,i)$ and $n$ {\bf downward arrows} of form $(k,i)\to (k-1,i+1)$.
For every Dyck path $\gamma$ in $\mathcal D_n$ define the {\bf directed subgraph $\mathcal D_n(\gamma)$} as the subgraph of $\mathcal D_n$ that lies below $\gamma$.
We label the upward arrows in $\gamma$ by $u_{1,3},u_{1,4},\dots,u_{1,n+2}$ (in this order) and the downward arrows in $\gamma$ by $u_{2(n+3)},u_{3(n+3)},\dots,u_{n+1(n+3)}$.
Faces of $\mathcal D_n$ that are supported on the subgraph $\mathcal D_n(\gamma)$ are called {\bf faces of $\mathcal D_n(\gamma)$}.
Finally, given $n$ for every $d\le n$ define the Dyck path $\gamma_d$ as
\begin{eqnarray}\label{eq:gamma d}
\begin{split}
\boxed{1}\to &q_{1,1}\to q_{2,1}\to\dots\to q_{d,1} \to q_{d-1,2}\to q_{d,2} \to q_{d-1,3}\to q_{d,3}\to \dots \\
\dots\to &q_{d-1,n-d+1}\to q_{d,n-d+1} \to q_{d-1,n-d+2}\to q_{d-2,n-d+3}\dots \to q_{1,n}\to \boxed{1}
\end{split}
\end{eqnarray}

\begin{Lemma}\label{lem:dyck for chords}
Let $d\in[n]$. There is a bijection between $\mathcal{C}h_d$ and 
\[
\{\text{faces }u_{ij} \text{ of }\mathcal D_n(\gamma_d)\} \cup \{\text{arrows in }\gamma_d\}.
\]
\end{Lemma}

\begin{Proposition}
    There is a bijection between Dyck paths in $\mathcal D_n$ and the set of chord diagram fans $\{\Sigma_{d,n}:1\le d\le n\}$ and their star subdivisions $\{\Star_{\tau_{i_1}}(\Star_{\tau_2}(\dots \Star_{\tau_{i_r}}(\Sigma_{d,n})\dots)):d\in[n], \{i_1,\dots,i_r\}\in\binom{[n-d+1]}{r}\}$. 
    More precisely, given $\gamma$ the associated fan is the inner normal fan of
    \begin{eqnarray}\label{eq:Dyck polytope}
    \mathcal P_\gamma = \operatorname{Newt}\left(\prod_{q_{k,i}\in\mathcal D_n(\gamma)} q_{k,i}\right).
    \end{eqnarray}
\end{Proposition}
\begin{proof}
    By Lemma~\ref{lem:dyck for chords} we have a Dyck path $\gamma_d$ for every chord diagram fan $\mathcal C_{d,n}$. 
    An arbitrary Dyck path $\gamma$ can be obtained from a Dyck path $\gamma_d$ by successively \emph{adding a face}. More precisely, this is achieved by replacing the arrows of the form
    \[
    q_{d,i}\to q_{d-1,i+1}\to q_{d,i+1} \quad \text{by} \quad q_{d,i}\to q_{d+1,i+1}\to q_{d,i+1}. 
    \]
    Assume a Dyck path $\gamma$ is obtained from a Dyck path $\gamma_d$ by the above replacement.
    Then $\mathcal D_n(\gamma)$ has one more face than $\mathcal D_n(\gamma_d)$, namely the face labelled by $u_{i+1,i+d+2}$.
    Adding the face $u_{i+1,i+d+2}$ corresponding to adding the chord $\overline{(i+1)(i+d+2)}$ to $\mathcal Ch_d$.
    This chord is added by the lift of the $(d+1)$-crossing $\overline{1(i+d+2)}$ and $\overline{(i+1)(n+3)}$.
    Hence, by Propositions~\ref{prop:stars} and \ref{prop: star of chord fan} the chord diagram fan associated to $\gamma$ is $\operatorname{Star}_{\tau}(\Sigma_{d,n})$ where $\tau$ is the two dimensional cone spanned by the rays corresponding to $\overline{1(i+d+2)}$ and $\overline{(i+1)(n+3)}$.
    Now \eqref{eq:Dyck polytope} follows from Proposition~\ref{prop:stars}.
\end{proof}

The polytopes in \eqref{eq:Dyck polytope} are exactly the polytopes associated to Dyck paths in \cite[Equation (8)]{CalvoCortesFrost2026}. Hence, Catalan pellytopes form a subclass of Dyck path polytopes. 
Their vertices are in bijection with chord diagrams using the alternative crossing rule from Remark~\ref{rmk:crossing chords}: given a Dyck path $\gamma$ define $\mathcal Ch_\gamma:=\{\overline{ab}:u_{ab}\in\mathcal D_n(\gamma)\}$. 
A {\bf chord diagram is $\gamma$-admissible} if it is a maximal collection of chords in $\mathcal Ch_\gamma$ so that each pair of chords satisfies one of the following
\begin{enumerate}
    \item the two chords do not cross, or
    \item they are of form $\overline{1j},\overline{i(n+3)}$ with $i\le j$ and $\overline{ij}\not\in\mathcal Ch_\gamma$.
\end{enumerate}

\begin{Corollary}\label{cor:Dyck fans}
Let $\gamma$ be a Dyck path in $\mathcal D_n$ and denote by $\Sigma_\gamma$ the inner normal fan of $\mathcal P_\gamma$ defined in \eqref{eq:Dyck polytope}. 
Then $\Sigma_\gamma$ is smooth and flag and its maximal cones correspond to $\gamma$-admissible chord diagrams.   
\end{Corollary}
\begin{proof}
The case of $\gamma=\gamma_d$ is Theorem~\ref{thm:chord fan} and Corollary~\ref{cor:smooth flag}.
The proofs of the aforementioned results generalize to star subdivisions of fans $\Sigma_{d,n}$ corresponding to crossings of chords. 
As each $\Sigma_\gamma$ is obtained from some $\Sigma_{d,n}$ by a sequence of such star subdivisions, the claim follows.
\end{proof}

\newcommand{\diagdots}[2][1]{%
  \pgfmathsetmacro{\dotangle}{#1*atan2(1,0.5*\horscale)}%
  \node[rotate=\dotangle] at (#2) {$\cdots$};%
}
\begin{figure}    \centering
\adjustbox{max width=\linewidth,center}{%
\begin{tikzpicture}
  
  \def\horscale{3}
  \coordinate (c11) at (0*\horscale,5);
  \coordinate (c21) at (-0.5*\horscale,4);  \coordinate (c22) at (0.5*\horscale,4);
  \coordinate (c31) at (-1*\horscale,3);  \coordinate (c32) at (0*\horscale,3);  \coordinate (c33) at (1*\horscale,3);
  \coordinate (c41) at (-1.5*\horscale,2);  \coordinate (c42) at (-0.5*\horscale,2);  \coordinate (c43) at (0.5*\horscale,2); \coordinate (c44) at (1.5*\horscale,2);  
  \coordinate (c51) at (-2*\horscale,1); \coordinate (c52) at (-1*\horscale,1); \coordinate (c53) at (0*\horscale,1);  \coordinate (c54) at (1*\horscale,1);  \coordinate (c55) at (2*\horscale,1);
  \coordinate (c61) at (-2.5*\horscale,0); \coordinate (c62) at (-1.5*\horscale,0);  \coordinate (c63) at (-0.5*\horscale,0);  \coordinate (c64) at (0.5*\horscale,0); \coordinate (c65) at (1.5*\horscale,0); \coordinate (c66) at (2.5*\horscale,0); 
  
  \node at (c11) {$q_{n,1}$};
  \node at (c21) {$q_{n-1,1}$};   \node at (c22) {$q_{n-1,2}$};
  \node at (c32) {$q_{n- 2,2}$};
  \node at (c41) {$q_{2,1}$};   \node at (c44) {$q_{2,n-1}$}; 
  \node at (c51) {$q_{1,1}$};   \node at (c52) {$q_{1,2}$};   \node at (c53) {$\cdots$};   \node at (c54) {$q_{1,n-1}$};   \node at (c55) {$q_{1,n}$} ;
  \node at (c61) {$\boxed{1}$};
  \foreach \i in {2,...,5}{
    \node at (c6\i) {1};
  }
  \node at (c66) {$\boxed{1}$};

  \draw[red, ->, shorten >= 0.3cm, shorten <= 0.6cm] (c21) -- node[midway, above left] {$u_{1,n+2}$} (c11);
  \draw[red, ->, shorten >= 0.6cm, shorten <= 0.3cm] (c31) -- node[midway, above left] {$u_{1,n+1}$} (c21);
  \draw[red, ->, shorten >= 0.3cm, shorten <= 0.3cm] (c41) -- node[midway, above left] {$u_{1,5}$} (c31);
  \draw[red, ->, shorten >= 0.3cm, shorten <= 0.3cm] (c51) -- node[midway, above left] {$u_{1,4}$} (c41);
  \draw[red, ->, shorten >= 0.3cm, shorten <= 0.3cm] (c61) -- node[midway, above left] {$u_{1,3}$} (c51);
  \draw[red, ->, shorten >= 0.6cm, shorten <= 0.3cm] (c11) -- node[midway, above right] {$u_{2,n+3}$} (c22);
  \draw[red, ->, shorten >= 0.3cm, shorten <= 0.6cm] (c22) -- node[midway, above right] {$u_{3,n+3}$} (c33);
  \draw[red, ->, shorten >= 0.6cm, shorten <= 0.3cm] (c33) -- node[midway, above right] {$u_{n-1,n+3}$} (c44);
  \draw[red, ->, shorten >= 0.3cm, shorten <= 0.6cm] (c44) -- node[midway, above right] {$u_{n,n+3}$} (c55);
  \draw[red, ->, shorten >= 0.3cm, shorten <= 0.3cm] (c55) -- node[midway, above right] {$u_{n+1,n+3}$} (c66);

  \draw[->, shorten >= 0.3cm, shorten <= 0.3cm] (c62) -- (c52);
  \draw[->, shorten >= 0.3cm, shorten <= 0.3cm] (c52) -- (c42);
  \draw[->, shorten >= 0.6cm, shorten <= 0.3cm] (c42) -- (c32);
  \draw[->, shorten >= 0.6cm, shorten <= 0.6cm] (c32) -- (c22);
  \draw[->, shorten >= 0.6cm, shorten <= 0.6cm] (c21) -- (c32);
  \draw[->, shorten >= 0.3cm, shorten <= 0.6cm] (c32) -- (c43);
  \draw[->, shorten >= 0.6cm, shorten <= 0.3cm] (c43) -- (c54);
  \draw[->, shorten >= 0.3cm, shorten <= 0.6cm] (c54) -- (c65);
  \draw[->, shorten >= 0.3cm, shorten <= 0.3cm] (c41) -- (c52);
  \draw[->, shorten >= 0.3cm, shorten <= 0.3cm] (c52) -- (c63);
  \draw[->, shorten >= 0.6cm, shorten <= 0.3cm] (c64) -- (c54);
  \draw[->, shorten >= 0.6cm, shorten <= 0.6cm] (c54) -- (c44);
  \draw[->, shorten >= 0.3cm, shorten <= 0.3cm] (c51) -- (c62);
  \draw[->, shorten >= 0.3cm, shorten <= 0.3cm] (c65) -- (c55);

  \node at ($(c11)!0.5!(c32)$) {$u_{2,n+2}$};
  \node at ($(c21)!0.5!(c42)$) {$u_{2,n+1}$};
  \node at ($(c31)!0.5!(c52)$) {$u_{2,5}$};
  \node at ($(c41)!0.5!(c62)$) {$u_{2,4}$};
  \node at ($(c22)!0.5!(c43)$) {$u_{3,n+2}$};
  \node at ($(c33)!0.5!(c54)$) {$u_{n-1,n+2}$};
  \node at ($(c44)!0.5!(c65)$) {$u_{n,n+2}$};
  \node at ($(c42)!0.5!(c63)$) {$u_{3,5}$};
  \node at ($(c43)!0.5!(c64)$) {$u_{n-1,n+1}$};

  \diagdots{c31}  \diagdots{c42}  \diagdots[-1]{c33}  \diagdots[-1]{c43}
\end{tikzpicture}
}
\caption{The oriented graph $\mathcal D_n$ with its vertex and face labeling. The Dyck path $\gamma_n$ is depicted in red}
 \label{fig:dyck diagram}
\end{figure}
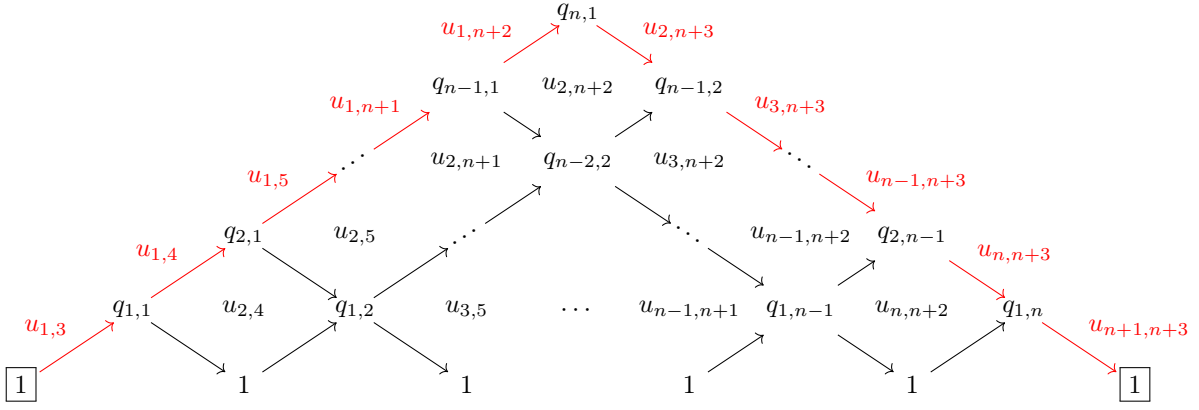

Fix a Dyck path $\gamma$ in $\mathcal D_n$. The diagram $\mathcal D_n(\gamma)$ induces an isomorphism between two Laurent polynomial rings:
\[
S_\gamma^{\pm}:=\mathbb C[u_{i,j}^{\pm 1}: u_{i,j}\in \mathcal D_n(\gamma)] \quad \text{and} \quad R_\gamma^{\pm}:=\mathbb C[y_1^{\pm 1},\dots,y_n^{\pm 1},q_{k,i}^{\pm 1}:q_{k,i}\in \mathcal D_n(\gamma)]
\]
The maps are induced by Berenstein--Fomin--Zelevinsky's chamber Ansatz \cite[\S2]{parametrizations} as we explain subsequently. 
Calvo Cortes and Frost define a map 
\begin{eqnarray}\label{eq:f gamma}
    f_\gamma: S_\gamma^{\pm} \to R_\gamma^{\pm} 
\end{eqnarray}
as follows: the variable $u_{a,b}$ labeling a face in $\mathcal D_n(\gamma)$ is mapped to the ratio of the product of polynomials labeling the vertices incident to $u$ above and below divided by the product of polynomials labeling the vertices to its right and its left, see \cite[Equation (6)]{CalvoCortesFrost2026}. Pictorially we have
\[
\begin{tikzpicture}[baseline=(current bounding box.center)]
  \def\horscale{3}
  \coordinate (c11) at (0,1);
  \coordinate (c21) at (-0.5*\horscale,0); \coordinate (c22) at (0.5*\horscale,0);
  \coordinate (c31) at (0,-1);

  \node at (c11) {$q_{k+1,i}$};
  \node at (c21) {$q_{k,i}$};   \node at (c22) {$q_{k,i+1}$};
  \node at (c31) {$q_{k-1,i+1}$};

  \draw[->, shorten >= 0.6cm, shorten <= 0.3cm] (c21) -- (c11);
  \draw[->, shorten >= 0.6cm, shorten <= 0.6cm] (c11) -- (c22);
  \draw[->, shorten >= 0.6cm, shorten <= 0.3cm] (c21) -- (c31);
  \draw[->, shorten >= 0.6cm, shorten <= 0.6cm] (c31) -- (c22);

  \node at ($(c11)!0.5!(c31)$) {$u_{i+1,i+k+2}$};
\end{tikzpicture}\quad \quad \quad
f_{\gamma}: u_{i+1,i+k+2} \mapsto \frac{q_{k+1,i}q_{k-1,i+1}}{q_{k,i}q_{k,i+1}}
\]
This is precisely the chamber Ansatz as in \cite[Equation 1.10 and Figure 2]{parametrizations}.
The case of $u_{1,k+3}$ and $u_{i,n+3}$ labelling arrows correspond in $\mathcal D_n(\gamma_d)$ to 
\[
\begin{tikzpicture}[baseline=(current bounding box.center)]
  \def\horscale{3}
  \coordinate (c11) at (0.5*\horscale,1);
  \coordinate (c21) at (0,0);

  \node at (c11) {$q'$};
  \node at (c21) {$q$};

  \draw[red, ->, shorten >= 0.3cm, shorten <= 0.3cm] (c21) -- node[midway, above left] {$u_{1,k+3}$} (c11);
\end{tikzpicture} \quad\quad f_{\gamma}: u_{1,k+3}\mapsto \frac{q}{q'} \quad\quad \text{and}\quad \quad
\begin{tikzpicture}[baseline=(current bounding box.center)]
  \def\horscale{3}
  \coordinate (c11) at (0,1);
  \coordinate (c21) at (0.5*\horscale,0);

  \node at (c11) {$q$};
  \node at (c21) {$q'$};

  \draw[red, ->, shorten >= 0.3cm, shorten <= 0.3cm] (c11) -- node[midway, above right] {$u_{k,n+3}$} (c21);
\end{tikzpicture} \quad\quad f_{\gamma}: u_{k,n+3}\mapsto \frac{y_{k-1}q}{q'}.
\]
Notice that the images of $u_{1,k+3}$ and $u_{k,n+3}$ depend on $\gamma$ while the others do not.

We proceed by defining a map in the opposite direction, $g_\gamma:R_\gamma^{\pm}\to S_\gamma^{\pm}$ resembling the dual chamber Ansatz as studied in \cite{Lara_Ghislain}.
Subsequently we show that the two maps are inverse.
For $q_{k,i}$ labeling a vertex in $\mathcal D_n(\gamma)$ define its image as the product of all inverses of $u$-variables contained in the maximal rectangle in $\mathcal D_n(\gamma)$ whose rightmost vertex is $q_{k,i}$. 
This includes the $u$-variables of all faces in the rectangle as well as all $u$-variables of upward arrows in the boundary of the rectangle.
Notice that the bottom vertex of such a rectangle is necessarily labeled by $1$.

\begin{Example}
For $n=4$ and $d=3$ the diagram $\mathcal D_4(\gamma_3)$ is
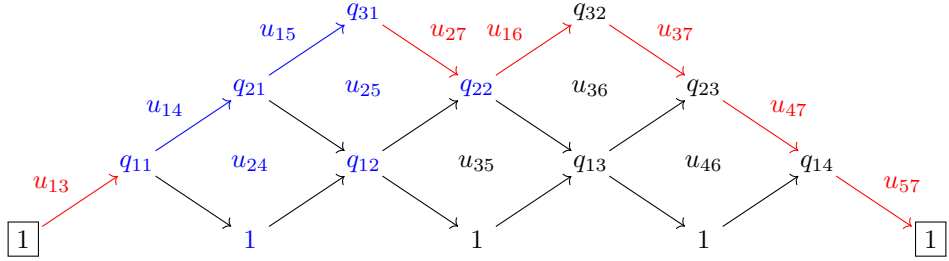
\[
\begin{tikzpicture}
  \def\horscale{3}
  \coordinate (c14) at (-0.5*\horscale,3);  \coordinate (c16) at (0.5*\horscale,3);
  \coordinate (c23) at (-1*\horscale,2);  \coordinate (c25) at (0*\horscale,2);  \coordinate (c27) at (1*\horscale,2);
  \coordinate (c32) at (-1.5*\horscale,1);  \coordinate (c34) at (-0.5*\horscale,1);  \coordinate (c36) at (0.5*\horscale,1);  \coordinate (c38) at (1.5*\horscale,1);
  \coordinate (c41) at (-2*\horscale,0);  \coordinate (c43) at (-1*\horscale,0);  \coordinate (c45) at (0*\horscale,0);  \coordinate (c47) at (1*\horscale,0);  \coordinate (c49) at (2*\horscale,0);

  \node[blue] at (c14) {$q_{31}$};   \node at (c16) {$q_{32}$};
  \node[blue] at (c23) {$q_{21}$};   \node[blue] at (c25) {$q_{22}$};   \node at (c27) {$q_{23}$};
  \node[blue] at (c32) {$q_{11}$};   \node[blue] at (c34) {$q_{12}$};   \node at (c36) {$q_{13}$};   \node at (c38) {$q_{14}$};
  \node at (c41) {$\boxed{1}$};   \node[blue] at (c43) {$1$};   \node at (c45) {$1$};   \node at (c47) {$1$};   \node at (c49) {$\boxed{1}$};

  \node[blue] at ($(c14)!0.5!(c34)$) {$u_{25}$};
  \node at ($(c16)!0.5!(c36)$) {$u_{36}$};
  \node[blue] at ($(c23)!0.5!(c43)$) {$u_{24}$};
  \node at ($(c25)!0.5!(c45)$) {$u_{35}$};
  \node at ($(c27)!0.5!(c47)$) {$u_{46}$};

  \draw[red, ->, shorten >= 0.3cm, shorten <= 0.3cm] (c41) -- node[midway, above left] {$u_{13}$} (c32);
  \draw[blue, ->, shorten >= 0.3cm, shorten <= 0.3cm] (c32) -- node[midway, above left] {$u_{14}$} (c23);
  \draw[blue, ->, shorten >= 0.3cm, shorten <= 0.3cm] (c23) -- node[midway, above left] {$u_{15}$} (c14);
  \draw[red, ->, shorten >= 0.3cm, shorten <= 0.3cm] (c14) -- node[midway, above right] {$u_{27}$} (c25);
  \draw[red, ->, shorten >= 0.3cm, shorten <= 0.3cm] (c25) -- node[midway, above left] {$u_{16}$} (c16);
  \draw[red, ->, shorten >= 0.3cm, shorten <= 0.3cm] (c16) -- node[midway, above right] {$u_{37}$} (c27);
  \draw[red, ->, shorten >= 0.3cm, shorten <= 0.3cm] (c27) -- node[midway, above right] {$u_{47}$} (c38);
  \draw[red, ->, shorten >= 0.3cm, shorten <= 0.3cm] (c38) -- node[midway, above right] {$u_{57}$} (c49);

  \draw[->, shorten >= 0.3cm, shorten <= 0.3cm] (c32) -- (c43);
  \draw[->, shorten >= 0.3cm, shorten <= 0.3cm] (c43) -- (c34);
  \draw[->, shorten >= 0.3cm, shorten <= 0.3cm] (c34) -- (c45);
  \draw[->, shorten >= 0.3cm, shorten <= 0.3cm] (c45) -- (c36);
  \draw[->, shorten >= 0.3cm, shorten <= 0.3cm] (c36) -- (c47);
  \draw[->, shorten >= 0.3cm, shorten <= 0.3cm] (c47) -- (c38);
  \draw[->, shorten >= 0.3cm, shorten <= 0.3cm] (c23) -- (c34);
  \draw[->, shorten >= 0.3cm, shorten <= 0.3cm] (c34) -- (c25);
  \draw[->, shorten >= 0.3cm, shorten <= 0.3cm] (c25) -- (c36);
  \draw[->, shorten >= 0.3cm, shorten <= 0.3cm] (c36) -- (c27);
\end{tikzpicture}
\]
The vertices and face variables of the maximal rectangle in $\mathcal D_4(\gamma_3)$ with rightmost corner $q_{22}$ are colored blue. 
Thus $g_{\gamma_3}(q_{22})=u_{14}^{-1}u_{15}^{-1}u_{24}^{-1}u_{25}^{-1}$.
\end{Example}

The image of a variable $y_k$ under $g_\gamma$ is given as follows: to a down arrow (labeled by some $u_{i(n+3)}$) associate the maximal rectangle in $\mathcal D_n(\gamma)$ whose top right edge is the down arrow. In particular, its bottom vertex is $(0,i)$. Denote this rectangle by $R(i,n+3)$.
Similarly, the vertex $(0,i-1)$ is the bottom vertex of a rectangle whose top left edge is the upward arrow labeled $u_{1,i+2}$. Denote this rectangle by $R(1,i+2)$.
Then $y_k$ maps to 
\[
\frac{\prod_{u_{a,b}\in R(k+1,n+3)} u_{a,b}}{\prod_{u_{a,b}\in R(1,k)} u_{a,b}}.
\]
Notice that $ R(1,k)$ is the rectangle determining the image of $q_{1,k}$ under $g_\gamma$.

\begin{Example}
Continuing with $\mathcal D_4(\gamma_3)$ we depict the rectangle $R(k+1,n+3)$ determining the numerator of $g_{\gamma_3}(y_2)$ in blue
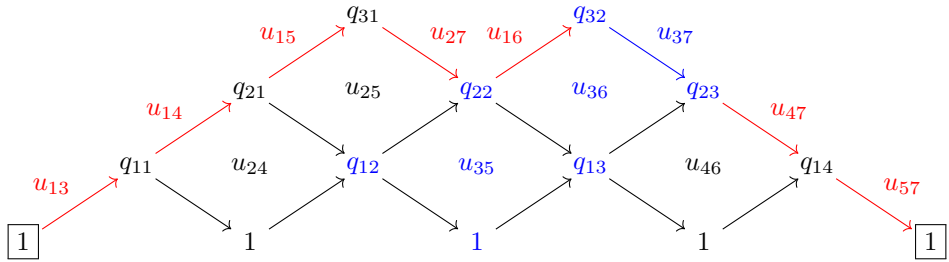
\[
\begin{tikzpicture}
  \def\horscale{3}
  \coordinate (c14) at (-0.5*\horscale,3);  \coordinate (c16) at (0.5*\horscale,3);
  \coordinate (c23) at (-1*\horscale,2);  \coordinate (c25) at (0*\horscale,2);  \coordinate (c27) at (1*\horscale,2);
  \coordinate (c32) at (-1.5*\horscale,1);  \coordinate (c34) at (-0.5*\horscale,1);  \coordinate (c36) at (0.5*\horscale,1);  \coordinate (c38) at (1.5*\horscale,1);
  \coordinate (c41) at (-2*\horscale,0);  \coordinate (c43) at (-1*\horscale,0);  \coordinate (c45) at (0*\horscale,0);  \coordinate (c47) at (1*\horscale,0);  \coordinate (c49) at (2*\horscale,0);

  \node at (c14) {$q_{31}$};   \node[blue] at (c16) {$q_{32}$};
  \node at (c23) {$q_{21}$};   \node[blue] at (c25) {$q_{22}$};   \node[blue] at (c27) {$q_{23}$};
  \node at (c32) {$q_{11}$};   \node[blue] at (c34) {$q_{12}$};   \node[blue] at (c36) {$q_{13}$};   \node at (c38) {$q_{14}$};
  \node at (c41) {$\boxed{1}$};   \node at (c43) {$1$};   \node[blue] at (c45) {$1$};   \node at (c47) {$1$};   \node at (c49) {$\boxed{1}$};

  \node at ($(c14)!0.5!(c34)$) {$u_{25}$};
  \node[blue] at ($(c16)!0.5!(c36)$) {$u_{36}$};
  \node at ($(c23)!0.5!(c43)$) {$u_{24}$};
  \node[blue] at ($(c25)!0.5!(c45)$) {$u_{35}$};
  \node at ($(c27)!0.5!(c47)$) {$u_{46}$};

  \draw[red, ->, shorten >= 0.3cm, shorten <= 0.3cm] (c41) -- node[midway, above left] {$u_{13}$} (c32);
  \draw[red, ->, shorten >= 0.3cm, shorten <= 0.3cm] (c32) -- node[midway, above left] {$u_{14}$} (c23);
  \draw[red, ->, shorten >= 0.3cm, shorten <= 0.3cm] (c23) -- node[midway, above left] {$u_{15}$} (c14);
  \draw[red, ->, shorten >= 0.3cm, shorten <= 0.3cm] (c14) -- node[midway, above right] {$u_{27}$} (c25);
  \draw[red, ->, shorten >= 0.3cm, shorten <= 0.3cm] (c25) -- node[midway, above left] {$u_{16}$} (c16);
  \draw[blue, ->, shorten >= 0.3cm, shorten <= 0.3cm] (c16) -- node[midway, above right] {$u_{37}$} (c27);
  \draw[red, ->, shorten >= 0.3cm, shorten <= 0.3cm] (c27) -- node[midway, above right] {$u_{47}$} (c38);
  \draw[red, ->, shorten >= 0.3cm, shorten <= 0.3cm] (c38) -- node[midway, above right] {$u_{57}$} (c49);

  \draw[->, shorten >= 0.3cm, shorten <= 0.3cm] (c32) -- (c43);
  \draw[->, shorten >= 0.3cm, shorten <= 0.3cm] (c43) -- (c34);
  \draw[->, shorten >= 0.3cm, shorten <= 0.3cm] (c34) -- (c45);
  \draw[->, shorten >= 0.3cm, shorten <= 0.3cm] (c45) -- (c36);
  \draw[->, shorten >= 0.3cm, shorten <= 0.3cm] (c36) -- (c47);
  \draw[->, shorten >= 0.3cm, shorten <= 0.3cm] (c47) -- (c38);
  \draw[->, shorten >= 0.3cm, shorten <= 0.3cm] (c23) -- (c34);
  \draw[->, shorten >= 0.3cm, shorten <= 0.3cm] (c34) -- (c25);
  \draw[->, shorten >= 0.3cm, shorten <= 0.3cm] (c25) -- (c36);
  \draw[->, shorten >= 0.3cm, shorten <= 0.3cm] (c36) -- (c27);
\end{tikzpicture}
\]
From the diagram we read off $g_{\gamma_3}(y_2)=\frac{u_{37}u_{36}u_{35}}{u_{24}u_{14}}$.
\end{Example}

\begin{Proposition}\label{prop:chamber inverse}
For $\gamma$ a Dyck path in $\mathcal D_n$ we have    $f_{\gamma}$ and $g_{\gamma}$ are inverse to each other. 
\end{Proposition}\label{lem:gfq}

\begin{proof}
We show that (1) $f_\gamma\circ g_\gamma (q_{k,i})=q_{k,i}$ and (2) $f_\gamma\circ g_\gamma (y_k)=y_k$.

\medskip
\noindent
(1) Notice that $f_{\gamma}$ sends the product of two $u$-variables corresponding to adjacent faces, say $u_{i+1,i+k+2}$ and $u_{i+1,i+k+3}$,  to the ratio of the polynomials at the corners of the rectangle formed by the two faces (top times bottom divided by right times left). Pictorially, we have
\[
\begin{tikzpicture}[baseline=(current bounding box.center)]
  \def\horscale{3}
  \coordinate (c11) at (0,3);
  \coordinate (c21) at (-0.5*\horscale,2);  \coordinate (c22) at (0.5*\horscale,2);
  \coordinate (c31) at (0,1);  \coordinate (c32) at (1*\horscale,1);
  \coordinate (c41) at (0.5*\horscale,0);

  \node at (c11) {$b$};
  \node at (c21) {$a$};   \node at (c22) {$c$};
  \node at (c31) {$d$};   \node at (c32) {$e$};
  \node at (c41) {$g$};

  \draw[->, shorten >= 0.3cm, shorten <= 0.3cm] (c21) -- (c11);
  \draw[->, shorten >= 0.3cm, shorten <= 0.3cm] (c11) -- (c22);
  \draw[->, shorten >= 0.3cm, shorten <= 0.3cm] (c21) -- (c31);
  \draw[->, shorten >= 0.3cm, shorten <= 0.3cm] (c31) -- (c22);
  \draw[->, shorten >= 0.3cm, shorten <= 0.3cm] (c22) -- (c32);
  \draw[->, shorten >= 0.3cm, shorten <= 0.3cm] (c31) -- (c41);
  \draw[->, shorten >= 0.3cm, shorten <= 0.3cm] (c41) -- (c32);

  \node at ($(c11)!0.5!(c31)$) {$u$};
  \node at ($(c22)!0.5!(c41)$) {$u'$};
\end{tikzpicture}
\quad \quad \Rightarrow\quad  f_{\gamma}(uu')=\frac{bd}{ac}\frac{cg}{de}=\frac{bg}{ae}.
\]      
Hence, the product of all (inverses of) face variables inside the rectangle corresponding to the image $g_{\gamma}(q_{k,i})$ is $\frac{q_{k,i}q'}{q''}$ where $q'$ corresponds to the bottom vertex and $q''$ the top vertex of the rectangle.
    
Similarly, the product of $u$-variables labeling a chain of upward arrows has image the bottom vertex divided by the top vertex.
In particular, this cancels $\frac{q'}{q''}$ from the image of $g_{\gamma}(q_{k,i})$, so that $f_\gamma(g_{\gamma}(q_{k,i}))=q_{k,i}$.

\noindent
(2) We use the same argument as above to see that 
\[\textstyle
f_{\gamma}\left({\prod_{u_{a,b}\in R(k+1,n+3)} u_{a,b}}\right) = \frac{y_k}{q_{1,k}}.
\]
The rectangle in the denominator corresponds to the image of $q_{1,k}^{-1}$ so we have $f_{\gamma}\left( \prod_{u_{a,b}\in R(1,k)} u_{a,b} \right)=q_{1,k}^{-1}$ and the claim follows.
\end{proof}

\begin{Remark}
    In the special case of the Dyck paths $\gamma_d$ in $\mathcal D_n$ we adopt the notation $S_{\gamma_d}^{\pm}=S_{d,n}^{\pm}$, $R_{\gamma_d}^{\pm}=R^{\pm}_{d,n}$ and $f_{\gamma_d}=f_{d,n}$,  $g_{\gamma_d}=g_{d,n}$.
\end{Remark}

We end this section with an example that shows the importance of considering polynomials $q_{k,i}$ associated to a Dyck path, as otherwise the resulting normal fans may not be simplicial.

\begin{Example}
For $n=3$ take the polynomial $f=q_{3,1}\prod_{j=1}^3q_{1,j}$, so
\[
f=(1+y_1)(1+y_2)(1+y_3)(1+y_1+y_1y_2+y_1y_2y_3)
\]
The corresponding Newton polytope has 12 vertices and its normal fan has nine rays generated by
\[
e_1,e_2,e_3,-e_1,-e_2,-e_3,e_1-e_2,e_1-e_3,e_2-e_3.
\]
There are two maximal cones, corresponding to the vertices $(0,0,1)$ and $(1,2,2)$ that have four rays each:
\[
\cone(e_1,e_2,e_1-e_3,e_2-e_3) \quad \text{and} \quad \cone(-e_2,-e_3,e_1-e_2,e_1-e_3).
\]
In particular, $\Sigma\left(\operatorname{Newt}(f)\right)$ is not smooth.
\end{Example}

\section{Application: binary geometries}\label{sec:binary}

The polytopes $\mathcal P_{d,n}$ define binary geometries in the known cases of $d=1,2,n$.
It is therefore natural to ask whether this is true in general.
As we have seen, our polytopes are within the class of polytopes associated to Dyck paths studied by Calvo Cortes and Frost in \cite{CalvoCortesFrost2026}.
In \emph{loc.cit.} it was shown that these polytopes satisfy many requirements of binary geometries.

Given a flag simplicial complex $\Delta$ on $[N]$ we write $i\not\sim j$ for $i,j\in \Delta$ if $\{i,j\}\not \in \Delta$.
We associate to each $i\in[N]$ a polynomial in $\mathbb C[u_1,\dots,u_N]$ determining a {\bf $u$-equation} 
\begin{eqnarray}\label{eq:u-equations}
    R_i=u_i+\prod_{j\not\sim i}u_j^{a_{ij}}-1=0
\end{eqnarray}
for some integers $a_{ij}>0$.
We are interested in the $u$-equations obtained from $\Sigma_{d,n}$. 
By Theorem~\ref{thm:chord fan}, the rays of $\Sigma_{d,n}$ are labelled by the chords in $\mathcal Ch_d$ as in  \eqref{eq:def d-chords}.
In particular, the $u$-equations take the form:
\begin{eqnarray}\label{eq:u-eq pell catalan}
    R_{ij}=u_{ij}+\prod_{\overline{pq}\not\sim_d \overline{ij}}u_{{pq}}-1=0
\end{eqnarray}
for every $\overline{ij}\in\mathcal Ch_d$. 
Inside the polynomial ring $\mathbb C[u_{ij}:\overline{ij}\in \mathcal Ch_d]$ define the {\bf chord ideal} $\mathcal I_{d,n}$ as the ideal generated by $R_{ij}$ for all $\overline{ij}\in \mathcal Ch_d$. 
Sometimes it will be more convenient to denote the $u$-variables by rays in $\Sigma_{d,n}$ instead of chords using the bijection \eqref{eq:chords and rays}. So, equivalently $\mathcal I_{d,n}\subset \mathbb C[u_\rho:\rho\in\Sigma_{d,n}(1)]$.
Additionally, we consider the Laurent polynomial ring $\mathbb C[u_{ij}^{\pm 1}: {\overline{ij}}\in \mathcal Ch_d]$ with the ideal $\mathcal I_{d,n}^\pm$ generated by the same set as $\mathcal I_{d,n}$.
Recall the notation for simplicial complexes from \S\ref{sec:pre}.

\begin{Definition}\label{def:BinaryGeom}
\cite[Definition 2.1]{Lam::LectureNotes}
    Let $\Delta$ be a flag simplicial complex on $[N]$. A {\bf binary geometry} $\widetilde{U}$ for $\Delta$ is an  affine algebraic variety $\tilde{U}=\tilde{U}(\Delta) \subset \mathbb{C}^N$ cut out by $N$ equations of the form~\eqref{eq:u-equations},
satisfying the following properties:
\begin{enumerate}
    \item $\dim \widetilde{U} =  \max_{S \in \Delta} |S|$ 
    \item for $S \subset [N]$ the subvariety $\widetilde{U}_S = \widetilde{U} \cap \{ u \in \mathbb{C}^N : u_i = 0 \; \forall i \in S\}$ is non-empty if and only if $S \in \Delta$;
    \item if $\widetilde{U}_S$ is non-empty, it is irreducible of codimension $|S|$ in $\widetilde{U}$.
\end{enumerate} 
 {If $a_{ij}\in \{0,1\}$ for all $i$ and $j$ the binary geometry is called {\bf perfect}.}
\end{Definition}
Binary geometries of type $\mathtt A$ cluster algebras are perfect. 
Grassmannian cluster algebras give rise to non-perfect binary geometries that have $a_{ij}>1$ for some $i,j$.

Consider the ideal $J_{d,n}^\pm\subset S_{d,n}^\pm$ generated by $q_{k,j} - (1+y_j+y_jy_{j+1}+\dots+y_{j}y_{j+1}\cdots y_{j+k-1})$ for all $k\in[d]$ and $j\in[n-k+1]$.
The composition of $f_{d,n}$ defined in \eqref{eq:f gamma} with the quotient map $\pi: S_{d,n}^\pm\to S_{d,n}^\pm/J_{d,n}^\pm$ is denoted by
\begin{eqnarray}\label{eq:poly map}
    \tilde f_{d,n}: \mathbb C[u_\rho^{\pm 1}:\rho\in \Sigma_{d,n}(1)] \to S_{d,n}^\pm/J_{d,n}^\pm, \quad u_{\rho} \mapsto \pi\left(f_{d,n}(u_\rho)\right).
\end{eqnarray}
By \cite[Theorem 4.2]{CalvoCortesFrost2026} the inverse image of $J_{d,n}^{\pm}$ is the chord ideal $\mathcal I_{d,n}^\pm$, that is:
\[
f_{d,n}^{-1}(J^{\pm}_{d,n})=\mathcal I_{d,n}^{\pm}=\ker(\tilde f_{d,n}).
\]
Hence, the isomorphism from Proposition~\ref{prop:chamber inverse} descends, and the map $\tilde f_{d,n}$ induces an isomorphism 
\begin{eqnarray}\label{eq:iso quasiaffine}
    \mathbb C[u_\rho^{\pm 1}:\rho\in\Sigma_{d,n}(1)]/\mathcal I_{d,n}^{\pm}\cong S_{d,n}^\pm/J_{d,n}^{\pm}.
\end{eqnarray}
Moreover, also by \cite[Theorem 4.2]{CalvoCortesFrost2026},
the isomorphism descends to polynomial rings
\begin{eqnarray}\label{eq:isom quot}
\mathbb C[u_\rho:\rho\in\Sigma_{d,n}(1)]/\mathcal I_{d,n}\cong S_{d,n}/J_{d,n}.
\end{eqnarray}

One last ingredient is the following:
\begin{Lemma}[Lemma 2.5 in \cite{Pellytopes}]\label{lem:link}
    Let $\Delta$ be a pure flag complex on $[n]$, and let  $\widetilde{U}$ be an affine algebraic variety $\widetilde{U} \subset \mathbb{C}^n$ cut out by $n$ equations of the form as in \eqref{eq:u-equations}.
For $S \in \Delta $, the subvariety $\widetilde{U}_S = \widetilde{U} \cap \{ u \in \mathbb{C}^n \mid u_k = 0 : k \in S\}$ is cut out by the equations:
\begin{align*}
    u_k &= 0 \quad \text{for } k  \in S, \qquad u_j = 1 \quad \text{for } j \notin W_S,\\
R_i &= \; u_i\; + \prod_{\ell \in W_S\colon \; \ell \nsim  i} u_\ell^{a_{i\ell}} - 1 = 0  \quad  \text{ for } i \in W_S.
\end{align*}
Moreover, the following hold: (1) If $\widetilde{U} $ is a binary geometry and $S \in \Delta$, then $\widetilde{U}_S$ is a binary geometry with underlying simplicial complex $\link_\Delta S$. (2) If $\widetilde{U}$ is an irreducible variety of dimension $\max_{S \in \Delta} \# S$, and $\widetilde{U}_{\{k\}}$ is a binary geometry with underlying simplicial complex $\link_\Delta\{k\}$ for each $k \in [n]$, then $\widetilde{U} $ is a binary geometry.
\end{Lemma}

These are all the necessary ingredients to verify that each Catalan pellytope defines a perfect binary geometry:

\begin{Proposition}\label{prop:binary}
The chord diagram fans $\Sigma_{d,n}$ define perfect binary geometries $\tilde{ \mathcal{U}}_{d,n}$ for $1\le d\le n$.
\end{Proposition}

\begin{proof}
By Corollary~\ref{cor:smooth flag} the fan $\Sigma_{d,n}$ defines a flag simplicial complex. 
Let $\tilde{ \mathcal{U}}_{d,n}$ be the affine variety defined by the $u$-equations in \eqref{eq:u-eq pell catalan}.
We need to verify (1), (2), and (3) from Definition~\ref{def:BinaryGeom}.

For (1), consider the quasi-affine variety $\mathcal U_{d,n}\subset (\mathbb C^*)^n$ defined by
\begin{align*}
\mathcal U_{d,n} := \left\{ x \in (\mathbb{C}^*)^{n} : q_{k,i}(x) \neq 0 \; \forall k\in [d],i\in[n-k+1] \right\}.
\end{align*}
Notice that it is of dimension $n$ and by \eqref{eq:iso quasiaffine} isomorphic to $\tilde{\mathcal{U}}_{d,n}\cap (\mathbb C^*)^N$, where $N=|\Sigma_{d,n}(1)|=n+\sum_{j=n-d+1}^n j$, by Proposition~\ref{cor:rays}. 
As $\tilde{\mathcal{U}}_{d,n}$ is the affine closure of the latter, it has the correct dimension.
Moreover, the ideal of $\tilde{\mathcal{U}}_{d,n}$, denoted $\mathcal{I}_{d,n}$ above, is the kernel of $\tilde f_{d,n}$ and therefore prime. 
Hence, $\tilde{\mathcal{U}}_{d,n}$ is irreducible.
It follows from \cite[Proof of Theorem 4.2]{CalvoCortesFrost2026} that $\mathcal I_{d,n} $ is generated by the $u$-equations.
    
To verify (2) and (3) hold we proceed by induction on $n$.
The base case is given by $n\le 3$: for $\Sigma_{d,n}$ with $d\le n\le 3$ the result is known by \cite{AHL::Stringy,Pellytopes}.
Now assume $n\ge 4$ and fix a ray $\rho\in\Sigma_{d,n}(1)$, and consider $\tilde U_{\rho} = \tilde{\mathcal U}_{d,n}\cap\{u_\rho=0\}$.
Let $W_\rho= \{
\eta\in\Sigma_{d,n}(1)\setminus\{\rho\}: \eta\sim_d\rho\}$
be the vertex set of $\link_{\Sigma_{d,n}}(\rho)$.
By Lemma~\ref{lem:link}, the equations defining $\widetilde U_\rho$ are
equivalent to
$u_\rho=0$, $u_\eta=1$ for $\eta\notin W_\rho\cup\{\rho\}$, together with
\[
u_\eta+ \prod_{\lambda\in W_\rho,\lambda\not\sim_d\eta} u_\lambda =1,
\quad \text{ for } \eta\in W_\rho.
\]
These are precisely the $u$-equations associated with the link
$\operatorname{link}_{\Sigma_{d,n}}(\rho)$. Hence $\tilde U_\rho
\cong
\tilde U(
\link_{\Sigma_{d,n}}(\rho)
)$.
By Proposition~\ref{prop:link}, there are integers $n_1,n_2\geq0$ with $n_1+n_2=n-1$ and for each $i\in\{1,2\}$ either $(d_i,n_i)=(0,0)$ or $d_i\in[n_i]$ such that
\[
\link_{\Sigma_{d,n}}(\rho) \cong \Sigma_{d_1,n_1}\times \Sigma_{d_2,n_2}.
\]
Rays belonging to different factors of this fan are compatible.
Therefore, the $u$-equations split into two independent systems, and
we obtain
\[
\tilde U_\rho \cong \tilde{\mathcal U}_{d_1,n_1}
\times
\tilde{\mathcal U}_{d_2,n_2},
\]
where $\tilde{\mathcal U}_{0,0}$ is a point.
Since $n_1,n_2<n$, both factors are binary geometries by induction. Their product is therefore a binary geometry for $\link_{\Sigma_{d,n}}(\rho)$.
In particular, $\tilde U_\rho$ is nonempty, irreducible, and of dimension 
$n-1$. Hence, by Lemma~\ref{lem:link}(2) $\tilde{\mathcal U}_{d,n}$ is a binary geometry.

The binary geometry is perfect as all exponents in \eqref{eq:u-eq pell catalan} are either $0$ or $1$.
\end{proof}

\begin{Remark}
Our setup is closely related to \cite{Telen_positive}.
Observe that the $q_{k,i}$'s only have positive coefficients, so the positive real orthant $\mathbb{R}^n_{>0}$ is a connected component of $\mathcal{U}_{d,n}(\mathbb R)$.
In fact, it follows from \emph{loc.cit.} that the positive part of $\mathcal U_{d,n}$ coincides with the positive part of the toric variety associated to $\mathcal P_{d,n}$.  
\end{Remark}

\bibliographystyle{plain} 
\bibliography{references} 

\end{document}

%% file: pictures.tex
\newcommand{\hexagon}[2][{{}}]{  
\begin{scope}[scale=0.6]
    \def\chords{{
        "1,3", 
        "1,4", 
        "1,5", 
        "1,6", 
        "2,4", 
        "2,5", 
        "2,6", 
        "2,7", 
        "3,5", 
        "3,6", 
        "3,7", 
        "4,6", 
        "4,7", 
        "5,7"  
      }}

    \foreach \vtx in {1,...,6}{
      \coordinate (y\vtx) at (-\vtx*360/6:1) ;
    }

    \coordinate (y7) at (y1);
    
    \foreach \entry in #2{
      \pgfmathsetmacro{\chord}{\chords[\entry-1]}
      \pgfmathsetmacro{\from}{{\chord}[0]]}
      \pgfmathsetmacro{\to}{{\chord}[1]]}

      \draw[] (y\from)     -- (y\to);

    }
    
    \draw (y1)
    \foreach \vtx in {2,...,7}{
      -- (y\vtx)
      
    };
    
    \foreach \entry in #1{
      \pgfmathsetmacro{\chord}{\chords[\entry-1]}
      \pgfmathsetmacro{\from}{{\chord}[0]]}
      \pgfmathsetmacro{\to}{{\chord}[1]]}

      \draw[line width = 1mm, yellow, opacity=0.4] (y\from)     -- (y\to);

    }
    
    \draw (y1)
      \foreach \vtx in {2,...,7}{
        -- (y\vtx)

      };

\end{scope}
}

\tikzset{
    edge/.style={
        shorten >=7pt,
        shorten <=7pt
    }
}

\newcommand\Angle{205}

\newcommand\picpone{
  \tikz[scale=0.4,
    x={( {cos(\Angle)*1cm}, {sin(\Angle)*0.5cm} )},
    y={( {-sin(\Angle)*1cm}, {cos(\Angle)*0.5cm} )},
    z={(0cm,0.75cm)} 
    ]{
\coordinate (e1) at (0.0, 0.0, 0.0);
\coordinate (e2) at (0.0, 0.0, 9);
\coordinate (e3) at (9, 0.0, 0.0);
\coordinate (e4) at (0.0, 9, 0.0);
\coordinate (e5) at (9, 9, 0.0);
\coordinate (e6) at (9, 0.0, 9);
\coordinate (e7) at (0.0, 9, 9);
\coordinate (e8) at (9, 9, 9);

\draw[edge,dashed] (e1) -- (e2);
\draw[edge,dashed] (e1) -- (e3);
\draw[edge] (e2) -- (e6);
\draw[edge] (e2) -- (e7);
\draw[edge] (e7) -- (e8);
\draw[edge] (e6) -- (e8);
\draw[edge] (e5) -- (e8);
\draw[edge] (e4) -- (e7);
\draw[edge] (e5) -- (e3);
\draw[edge] (e4) -- (e5);
\draw[edge,dashed] (e1) -- (e4);
\draw[edge] (e3) -- (e6);

    \begin{scope}[shift={(e1)}, x={(1cm,0cm)}, y={(0cm,1cm)}]
      \hexagon[{{}}]{{1,10,12}}
    \end{scope}
    \begin{scope}[shift={(e2)}, x={(1cm,0cm)}, y={(0cm,1cm)}]
      \hexagon[{{}}]{{7,10,12}}
    \end{scope}
    \begin{scope}[shift={(e3)}, x={(1cm,0cm)}, y={(0cm,1cm)}]
      \hexagon[{{}}]{{1,2,12}}
    \end{scope}
    \begin{scope}[shift={(e4)}, x={(1cm,0cm)}, y={(0cm,1cm)}]
      \hexagon[{{}}]{{3,7,10}}
    \end{scope}
    \begin{scope}[shift={(e5)}, x={(1cm,0cm)}, y={(0cm,1cm)}]
      \hexagon[{{}}]{{1,2,3}}
    \end{scope}
    \begin{scope}[shift={(e6)}, x={(1cm,0cm)}, y={(0cm,1cm)}]
      \hexagon[{{}}]{{2,7,12}}
    \end{scope}
    \begin{scope}[shift={(e7)}, x={(1cm,0cm)}, y={(0cm,1cm)}]
      \hexagon[{{}}]{{1,3,9}}
    \end{scope}
    \begin{scope}[shift={(e8)}, x={(1cm,0cm)}, y={(0cm,1cm)}]
      \hexagon[{{}}]{{7,2,3}}
      \end{scope}
      \draw[line width = 4, yellow, opacity=0.4] (e7)     -- (e8);
      \draw[line width = 4, Mulberry, opacity=0.4] (e4)     -- (e7);
      \draw[line width = 4, SkyBlue, opacity=0.4] (e6)     -- (e8);
}
}

\newcommand\picptwo{
  \tikz[scale=0.4,
    x={( {cos(\Angle)*1cm}, {sin(\Angle)*0.5cm} )},
    y={( {-sin(\Angle)*1cm}, {cos(\Angle)*0.5cm} )},
    z={(0cm,0.75cm)} 
    ]{

\coordinate (e1) at (0.0, 0.0, 0.0);
\coordinate (e2) at (0.0, 0.0, 9);
\coordinate (f1) at (0.0, 6, 0.0);
\coordinate (f2) at (0.0, 6, 9);
\coordinate (f3) at (3.5, 9, 0.0);
\coordinate (f4) at (3.5, 9, 9);
\coordinate (f5) at (6, 0.0, 9);
\coordinate (f6) at (6, 9, 9);
\coordinate (e3) at (9, 0.0, 0.0);
\coordinate (f7) at (9, 0.0, 6);
\coordinate (e4) at (9, 9, 0.0);
\coordinate (f8) at (9, 9, 6);

\draw[edge] (f2) -- (e2);
\draw[edge] (f1) -- (f2);
\draw[edge, dashed] (e2) -- (e1);
\draw[edge, dashed] (f1) -- (e1);
\draw[edge] (f2) -- (f4);
\draw[edge] (f5) -- (e2);
\draw[edge] (f4) -- (f6);
\draw[edge] (f5) -- (f6);
\draw[edge] (f1) -- (f3);
\draw[edge] (f3) -- (f4);
\draw[edge] (f7) -- (e3);
\draw[edge] (f7) -- (f5);
\draw[edge, dashed] (e3) -- (e1);
\draw[edge] (f3) -- (e4);
\draw[edge] (e4) -- (e3);
\draw[edge] (f8) -- (e4);
\draw[edge] (f8) -- (f6);
\draw[edge] (f7) -- (f8);

    \begin{scope}[shift={(e1)}, x={(1cm,0cm)}, y={(0cm,1cm)}]
      \hexagon[{{}}]{{1,10,12}}
    \end{scope}
    \begin{scope}[shift={(e2)}, x={(1cm,0cm)}, y={(0cm,1cm)}]
      \hexagon[{{}}]{{7,10,12}}
    \end{scope}
    \begin{scope}[shift={(e3)}, x={(1cm,0cm)}, y={(0cm,1cm)}]
      \hexagon[{{}}]{{1,2,12}}
    \end{scope}
    \begin{scope}[shift={(e4)}, x={(1cm,0cm)}, y={(0cm,1cm)}]
      \hexagon[{{}}]{{1,2,3}}
    \end{scope}
    \begin{scope}[shift={(f1)}, x={(1cm,0cm)}, y={(0cm,1cm)}]
      \hexagon[{{}}]{{1,9,10}}
    \end{scope}
    \begin{scope}[shift={(f2)}, x={(1cm,0cm)}, y={(0cm,1cm)}]
      \hexagon[{{}}]{{7,9,10}}
    \end{scope}
    \begin{scope}[shift={(f3)}, x={(1cm,0cm)}, y={(0cm,1cm)}]
      \hexagon[{{}}]{{1,3,9}}
    \end{scope}
    \begin{scope}[shift={(f4)}, x={(1cm,0cm)}, y={(0cm,1cm)}]
      \hexagon[{{}}]{{3,7,9}}
    \end{scope}
    \begin{scope}[shift={(f5)}, x={(1cm,0cm)}, y={(0cm,1cm)}]
      \hexagon[{{}}]{{5,7,12}}
    \end{scope}
    \begin{scope}[shift={(f6)}, x={(1cm,0cm)}, y={(0cm,1cm)}]
      \hexagon[{{}}]{{3,5,7}}
    \end{scope}
    \begin{scope}[shift={(f7)}, x={(1cm,0cm)}, y={(0cm,1cm)}]
      \hexagon[{{}}]{{2,5,12}}
    \end{scope}
    \begin{scope}[shift={(f8)}, x={(1cm,0cm)}, y={(0cm,1cm)}]
      \hexagon[{{}}]{{2,3,5}}
    \end{scope}
    \draw[fill=Mulberry,opacity=0.2] (f1) -- (f2) -- (f4) -- (f3);
\draw[fill=SkyBlue,opacity=0.2] (f5) -- (f6) -- (f8) -- (f7);
\draw[line width = 4, yellow, opacity=0.4] (f4)     -- (f6);
}
}

\newcommand\picpthree{
  \tikz[scale=0.4,
    x={( {cos(\Angle)*1cm}, {sin(\Angle)*0.5cm} )},
    y={( {-sin(\Angle)*1cm}, {cos(\Angle)*0.5cm} )},
    z={(0cm,0.75cm)} 
    ]{

\coordinate (f1)  at (3,  6 , 9);
\coordinate (f2)  at (0 ,  3, 9);
\coordinate (f3)  at (6 ,  6 , 9);
\coordinate (e1)  at (0 ,  0 , 9);
\coordinate (f4)  at (6 ,  0 , 9);
\coordinate (f5)  at (6 ,  9, 6 );
\coordinate (f6)  at (6 ,  9, 0 );
\coordinate (f7)  at (0 ,  3 , 0 );
\coordinate (f8)  at (9,  9, 6 );
\coordinate (e2)  at (0 ,  0 , 0 );
\coordinate (f9)  at (12   ,  0 , 3 );
\coordinate (f10) at (12   ,  9, 3 );
\coordinate (e3)  at (12   ,  0 , 0 );
\coordinate (e4)  at (12   ,  9, 0 );

\draw[edge] (f1) -- (f2);
\draw[edge] (f1) -- (f3);
\draw[edge] (f2) -- (e1);
\draw[edge] (f3) -- (f4);
\draw[edge] (f4) -- (e1);
\draw[edge] (f5) -- (f6);
\draw[edge] (f1) -- (f5);
\draw[edge] (f7) -- (f6);
\draw[edge] (f7) -- (f2);
\draw[edge] (f5) -- (f8);
\draw[edge] (f3) -- (f8);
\draw[edge, dashed] (e1) -- (e2);
\draw[edge,dashed] (f7) -- (e2);
\draw[edge] (f9) -- (f4);
\draw[edge] (f8) -- (f10);
\draw[edge] (f9) -- (f10);
\draw[edge] (f9) -- (e3);
\draw[edge,dashed] (e3) -- (e2);
\draw[edge] (f6) -- (e4);
\draw[edge] (f10) -- (e4);
\draw[edge] (e4) -- (e3);
1 10 12
    \begin{scope}[shift={(e1)}, x={(1cm,0cm)}, y={(0cm,1cm)}]
      \hexagon[{{}}]{{7,10,12}}
    \end{scope}
    \begin{scope}[shift={(e2)}, x={(1cm,0cm)}, y={(0cm,1cm)}]
      \hexagon[{{}}]{{1,10,12}}
    \end{scope}
    \begin{scope}[shift={(e3)}, x={(1cm,0cm)}, y={(0cm,1cm)}]
      \hexagon[{{}}]{{1,2,12}}
    \end{scope}
    \begin{scope}[shift={(e4)}, x={(1cm,0cm)}, y={(0cm,1cm)}]
      \hexagon[{{}}]{{1,2,3}}
    \end{scope}
    \begin{scope}[shift={(f1)}, x={(1cm,0cm)}, y={(0cm,1cm)}]
      \hexagon[{{}}]{{6,7,9}}
    \end{scope}
    \begin{scope}[shift={(f2)}, x={(1cm,0cm)}, y={(0cm,1cm)}]
      \hexagon[{{}}]{{7,9,10}}
    \end{scope}
    \begin{scope}[shift={(f3)}, x={(1cm,0cm)}, y={(0cm,1cm)}]
      \hexagon[{{}}]{{5,6,7}}
    \end{scope}
    \begin{scope}[shift={(f4)}, x={(1cm,0cm)}, y={(0cm,1cm)}]
      \hexagon[{{}}]{{5,7,12}}
    \end{scope}
    \begin{scope}[shift={(f5)}, x={(1cm,0cm)}, y={(0cm,1cm)}]
      \hexagon[{{}}]{{3,6,9}}
    \end{scope}
    \begin{scope}[shift={(f6)}, x={(1cm,0cm)}, y={(0cm,1cm)}]
      \hexagon[{{}}]{{1,3,9}}
    \end{scope}
    \begin{scope}[shift={(f7)}, x={(1cm,0cm)}, y={(0cm,1cm)}]
      \hexagon[{{}}]{{1,9,10}}
    \end{scope}
    \begin{scope}[shift={(f8)}, x={(1cm,0cm)}, y={(0cm,1cm)}]
      \hexagon[{{}}]{{3,5,6}}
    \end{scope}
        \begin{scope}[shift={(f9)}, x={(1cm,0cm)}, y={(0cm,1cm)}]
      \hexagon[{{}}]{{2,5,12}}
    \end{scope}
    \begin{scope}[shift={(f10)}, x={(1cm,0cm)}, y={(0cm,1cm)}]
      \hexagon[{{}}]{{2,3,5}}
    \end{scope}
    \draw[fill=Mulberry,opacity=0.2] (f2) -- (f1) -- (f5) -- (f6) -- (f7);
\draw[fill=SkyBlue,opacity=0.2] (f8) -- (f3) -- (f4) -- (f9) -- (f10);
\draw[fill=yellow,opacity=0.2] (f1) -- (f5) -- (f8) -- (f3);
  }
}
